\documentclass[a4paper,12pt]{article}
\usepackage[utf8]{inputenc}
\usepackage[english]{babel}
\usepackage{sectsty}
\usepackage{mathtools,amssymb, amsthm, amsfonts, stmaryrd, eufrak,mathpazo}
\usepackage[backend=biber]{biblatex}
\usepackage{graphicx, subfig}
\usepackage{hhline}
\usepackage{rotating}
\usepackage{geometry}
\usepackage{titlesec}
\usepackage[absolute]{textpos}
\usepackage{tikz, tikz-qtree, tikz-cd}
\tikzset{
    lablv/.style={anchor=south, rotate=90, inner sep=.5mm},
    labld/.style={anchor=south, rotate=150, inner sep=.8mm},
    lablad/.style={anchor=north, rotate=40, inner sep=1.2mm}
}
\usepackage{pgffor,etoolbox}
\usepackage{pgfplotstable}
\usepackage{longtable}
\usepackage{booktabs}
\pgfplotsset{compat=1.18}
\usepackage{fp}
\usepackage{xstring}
\usepackage{youngtab ,ytableau}
\usepackage{enumitem}
\usepackage{fancyhdr}
\usepackage{colortbl}
\usepackage{color}
\definecolor{valentia}{rgb}{233,78,82}
\definecolor{titleblue}{rgb}{114, 146, 162}
\usepackage{array}
\usepackage{makecell}
\usepackage{float}
\usepackage{phoenician}
\usepackage{mdframed}
\usepackage{hyperref}
\usepackage{upgreek}
\usepackage{authblk}
\usepackage{multirow}
\usepackage{multicol}
\usepackage{scrextend}
\DeclareMathAlphabet{\mathscr}{U}{rsfso}{m}{n}
\DeclareMathSymbol{\smin}{\mathbin}{AMSa}{"39}
\usepackage{authblk}
\newcommand{\Addresses}{{
  \bigskip
  \footnotesize

  G.~Bellamy, \textsc{School of Mathematics and Statistics,
 University Place, Glasgow, G12 8QQ, Glasgow, UK.}\par\nopagebreak
  \textit{E-mail address}, G.~Bellamy: \texttt{gwyn.bellamy@glasgow.ac.uk} \par \nopagebreak
  \textit{Website}, G.~Bellamy: \texttt{https://www.gla.ac.uk/schools/mathematicsstatistics/staff/gwynbellamy/}

  \medskip

  R.~Paegelow, \textsc{Institut Montpelliérain Alexander Grothendieck, Université Montpellier 2,  Place Eugène Bataillon, 34095 MONTPELLIER Cedex, FRANCE.}\par\nopagebreak
  \textit{E-mail address}, R.~Paegelow: \texttt{raphael.paegelow@umontpellier.fr}

  \medskip

}}
\title{The Procesi bundle over the $\Gamma$-fixed points of the Hilbert scheme of points in $\mathbb{C}^2$}

\author{Gwyn Bellamy \& Raphaël Paegelow}
\begin{document}
\newgeometry{top=2cm, bottom=2.5cm, left=2.5cm, right=2.5cm}
\theoremstyle{definition}
\newtheorem{deff}{Definition}[section]
\newtheorem{nota}[deff]{Notation}

\newtheoremstyle{cremark}
    {\dimexpr\topsep/2\relax}
    {\dimexpr\topsep}
    {}
    {}
    {}
    {.}
    {.5em}
    {}

\theoremstyle{cremark}
\newtheorem{rmq}[deff]{Remark}
\newtheorem{ex}[deff]{Example}

\theoremstyle{plain}

\newtheorem{thm}[deff]{Theorem}
\newtheorem{cor}[deff]{Corollary}
\newtheorem{prop}[deff]{Proposition}
\newtheorem{lemme}[deff]{Lemma}
\newtheorem{fait}[deff]{Fact}
\newtheorem{conj}[deff]{Conjecture}
\newtheorem*{theorem1}{Theorem}
\newtheorem*{cor1}{Corollary}

\maketitle
\section*{Abstract}
For $\Gamma$ a finite subgroup of $\mathrm{SL}_2(\mathbb{C})$ and $n \geq 1$, we study the fibers of the Procesi bundle over the $\Gamma$-fixed points of the Hilbert scheme of $n$ points in the plane. For each irreducible component of this fixed point locus, our approach reduces the study of the fibers of the Procesi bundle, as an $(\mathfrak{S}_n \times \Gamma)$-module, to the study of the fibers of the Procesi bundle over an irreducible component of dimension zero in a smaller Hilbert scheme. When $\Gamma$ is of type $A$, our main result shows, as a corollary, that the fiber of the Procesi bundle over the monomial ideal associated with a partition $\lambda$ is induced, as an $(\mathfrak{S}_n \times \Gamma)$-module, from the fiber of the Procesi bundle over the monomial ideal associated with the core of $\lambda$. We give different proofs of this corollary in two edge cases using only representation theory and symmetric functions.

\section{Introduction}
The Procesi bundle is an important vector bundle on the Hilbert scheme of points in $\mathbb{C}^2$. It has played a key role in Hamain's proof of the theorem $n!$ \cite[Theorem $5.2.1$]{H01}. Take $n$ an integer greater or equal to $1$.  If $\mathfrak{S}_n$ denotes the symmetric group on $n$ letters, then the fibers of the Procesi bundle $\mathscr{P}^n$ are $\mathfrak{S}_n$-modules, isomorphic to the regular representation of $\mathfrak{S}_n$; thus the Procesi bundle has rank $n!$. Consider now the natural action $\mathrm{GL}_2(\mathbb{C})$ on $\mathbb{C}^2$. This action induces a $\mathrm{GL}_2(\mathbb{C})$-action on the Hilbert scheme $\mathcal{H}_n$ of $n$ points in $\mathbb{C}^2$. Let $\Gamma$ be a finite subgroup of $\mathrm{SL}_2(\mathbb{C})$. Over $\mathcal{H}_n^{\Gamma}$, the fibers of $\mathscr{P}^n$ are $(\mathfrak{S}_n \times \Gamma)$-modules. The main result of this article shows that, for each irreducible component $\mathcal{C}$ of $\mathcal{H}_n^{\Gamma}$ and for each $I \in \mathcal{C}$, the fiber of $\mathscr{P}^n$ at $I$ can be constructed, as an $(\mathfrak{S}_n \times \Gamma)$-module, by induction of the fiber of $\mathscr{P}^k$ over an ideal $I_0\in \mathcal{H}_{k}^{\Gamma}$ for some ${k \leq n}$ such that $\{I_0\}$ is an irreducible component of $\mathcal{H}_k^{\Gamma}$. The integer $k$ is explicit and depends on $\mathcal{C}$, $\Gamma$ and $n$. The result is stated in section \ref{sec_Proc_bund} below. But first we introduce the main players and the notation used in the article.

\subsection{The Hilbert scheme of points in the plane}
\noindent Let $\mathcal{H}_n$ be the Hilbert scheme of $n$ points in $\mathbb{C}^2$. As a set
\begin{equation*}
\mathcal{H}_n= \{I \subset \mathbb{C}[x,y]| I \text{ is an ideal  and } \mathrm{dim}(\mathbb{C}[x,y]/I)=n\}.
\end{equation*}
Fogarty showed \cite[Proposition $2.2$ \& Theorem $2.9$]{Fo68} that $\mathcal{H}_n$ is a smooth connected $2n$ dimensional algebraic variety.
Let the group $\mathfrak{S}_n$ act on ${(\mathbb{C}^2)}^n$ by permuting the $n$ copies of $\mathbb{C}^2$ and denote by $\sigma_n$ the Hilbert-Chow morphism. It is defined as follows
\begin{center}
$\sigma_n\colon \begin{array}{ccccc}
\mathcal{H}_n & \to & {(\mathbb{C}^2)}^n/\mathfrak{S}_n\\
I & \mapsto & \sum_{p \in V(I)}{\mathrm{dim}\left((\mathbb{C}[x,y]/I)_p\right)[p]}  \\
\end{array}$\\
\end{center}
where $[p]$ denotes the orbit of $p$ in  ${(\mathbb{C}^2)}^n/\mathfrak{S}_n$ for $p \in {(\mathbb{C}^2)}^n$.

\subsection{The Procesi bundle}
\label{sec_Proc_bund}
The Procesi vector bundle is a $GL_2(\mathbb{C})$-equivariant vector bundle on the Hilbert scheme of $n$ points in $\mathbb{C}^2$.
To construct the Procesi bundle, one first needs to introduce the isospectral Hilbert scheme. The $n^{\text{th}}$-isospectral Hilbert scheme, denoted by $\mathcal{X}_n$, is the reduced fiber product of $\mathcal{H}_n$ with $(\mathbb{C}^2)^n$ over $(\mathbb{C}^2)^n/\mathfrak{S}_n$:
\begin{equation}\label{eq:cartprod}
\begin{tikzcd}
 |[label={[label distance=-2mm]-45:\lrcorner}]|  \mathcal{H}_n \times_{{(\mathbb{C}^2)}^n/{\mathfrak{S}_n}} {(\mathbb{C}^2)}^n \ar[d, "\rho_n"'] \ar[r, "f_n"]   & {(\mathbb{C}^2)}^n \ar[d, "\pi_n"]\\
\mathcal{H}_n  \ar[r, "\sigma_n"']& {(\mathbb{C}^2)}^n/{\mathfrak{S}_n}
\end{tikzcd}
\end{equation}
Here, the morphism $\pi_n$ is the quotient map. The scheme $\mathcal{X}_n$ is an algebraic variety, projective over $(\mathbb{C}^2)^n$. Crucially, Haiman \cite[Theorem $5.2.1$]{H03} has proven that $\rho_n$ is a finite and flat morphism. This implies that the sheaf $\mathscr{P}^n:={\rho_n}_*\mathcal{O}_{\mathcal{X}_n}$ is locally free and thus defines a vector bundle on $\mathcal{H}_n$. This vector bundle is the $n^{\text{th}}$-Procesi bundle. Note that, by construction, the fibers of $\mathscr{P}^n$  are $\mathfrak{S}_n$-modules. The natural $\mathrm{GL}_2(\mathbb{C})$-action on $\mathcal{H}_n$ and the diagonal action on ${(\mathbb{C}^2)}^n$ give a $\mathrm{GL}_2(\mathbb{C})$-action on $\mathcal{X}_n$, making $\mathscr{P}^n$ into a $\mathrm{GL}_2(\mathbb{C})$-equivariant vector bundle. Moreover, by letting $\mathfrak{S}_n$ act trivially on $\mathcal{H}_n$, all morphisms $\rho_n$, $\sigma_n$, $\pi_n$ and $f_n$ are $(\mathfrak{S}_n \times \mathrm{GL}_2(\mathbb{C}))$-equivariant.

For $I \in \mathcal{H}_n$, denote by $\mathscr{P}^n_{|I}$ the fiber of the vector bundle associated with $\mathscr{P}^n$ at $I$. Note that when $I \in \mathcal{H}_n^{\Gamma}$, the fiber $\mathscr{P}^n_{|I}$ is an $(\mathfrak{S}_n \times \Gamma)$-module. Let $\mathcal{C}$ be an irreducible component of $\mathcal{H}_n^{\Gamma}$ of dimension $2r$. Take ${(p_1,\dots,p_r) \in {({(\mathbb{C}^2)}\setminus \{(0,0)\})}^{r}}$ such that for each $(i,j) \in \llbracket 1, r \rrbracket ^2, i \neq j \Rightarrow \Gamma p_i \cap \Gamma p_j = \emptyset$. Let $q:=(\Gamma p_1,\dots,\Gamma p_{r}) \in {(\mathbb{C}^2)}^{r\ell}$ and $p:=(0,q) \in {(\mathbb{C}^2)}^n$. Let $S_p$ denote the stabilizer of $p$ in $\mathfrak{S}_n \times \Gamma$. Let $\mathrm{g}_{\Gamma}:=n-\ell r$. Then there exists a unique irreducible component $\{I_{0}\}$ of the scheme $\mathcal{H}_{\mathrm{g}_{\Gamma}}^{\Gamma}$ such that a generic point of $\mathcal{C}$ is of the form $V(I_0) \cup V(q)$. This defines a bijection between the irreducible components of $\mathcal{H}_{\mathrm{g}_{\Gamma}}^{\Gamma}$ of dimension zero and the $2r$-dimensional components of $\mathcal{H}_n^{\Gamma}$.  The main result of this article is the following theorem.
\begin{theorem1}
For each irreducible component $\mathcal{C}$ of $\mathcal{H}_n^{\Gamma}$, there exists an isomorphism of groups ${\textphnc{\As} \colon S_p \xrightarrow{\,\smash{\raisebox{-0.65ex}{\ensuremath{\scriptstyle\sim}}}\,} \mathfrak{S}_{\mathrm{g}_{\Gamma}} \times \Gamma}$, making $\mathscr{P}^{\mathrm{g}_{\Gamma}}_{|I_{0}}$ into a $S_p$-module such that for each $I \in \mathcal{C}$,
\begin{equation*}
\left[\mathscr{P}^n_{|I}\right]_{\mathfrak{S}_n \times \Gamma}=\left[\mathrm{Ind}_{S_p}^{\mathfrak{S}_n\times \Gamma}\left(\mathscr{P}^{\mathrm{g}_{\Gamma}}_{|I_{0}}\right)\right]_{\mathfrak{S}_n \times \Gamma}.
\end{equation*}
\end{theorem1}
\noindent This theorem reduces the study of the fibers of the Procesi bundle over the $\Gamma$-fixed points  to the study of the fibers of the Procesi bundle over the irreducible components of the $\Gamma$-fixed points of dimension zero. For this reason, we refer throughout to this result as the reduction theorem.

\subsection{Combinatorial consequences in type $A$}
In the case where $\Gamma$ is the cyclic subgroup $\mu_{\ell}$ (of order $\ell$, generated by $\omega_{\ell}$) in the maximal diagonal torus of $\mathrm{SL}_2(\mathbb{C})$ (type $A$), the reduction theorem interplays well with the combinatorics of $\ell$-cores. Indeed, each irreducible component of the scheme $\mathcal{H}_n^{\Gamma}$ contains a fixed point under the maximal diagonal torus of $\mathrm{SL}_2(\mathbb{C})$. These fixed points are indexed by partitions of $n$. If $\lambda$ is a partition of $n$, let $I_{\lambda} \in \mathcal{H}_n$ denote the associated fixed point and write $\mathscr{P}^{n}_{\lambda}$ for the fiber of the vector bundle associated with $\mathscr{P}^n$ at $I_{\lambda}$. Denote by  $\gamma_{\ell}$ the $\ell$-core associated with $\lambda$. The size of $\gamma_{\ell}$ is denoted $\mathrm{g}_{\ell}$ and $r_{\ell}:=\frac{n-\mathrm{g}_{\ell}}{\ell}$.
Let $\tau_{\ell}$ be the character of $\mu_{\ell}$ such that $\tau_{\ell}(\omega_{\ell})$ is equal to $\zeta_{\ell}$, a fixed primitive $\ell^{\text{th}}$ root of unity, and by $w_{\ell,n} \in \mathfrak{S}_n$ the product of the $r_{\ell}$ cycles of length $\ell$:
\begin{equation*}
 (\mathrm{g}_{\ell}+1,...,\mathrm{g}_{\ell}+\ell)...(n-\ell+1,...,n).
\end{equation*}
Let $C_{\ell,n}$ be the cyclic subgroup of $\mathfrak{S}_{r_{\ell}\ell}$ generated by $w_{\ell,n}$. Consider also the subgroup $W_{\ell,n}^{\mathrm{g}_{\ell}}:= \mathfrak{S}_{\mathrm{g}_{\ell}} \times C_{\ell,n}$ of $\mathfrak{S}_n$. Denote by $\theta_{\ell}$ the character of $C_{\ell,n}$ such that $\theta_{\ell}(w_{\ell,n})=\zeta_{\ell}$. The following is a corollary of the main reduction theorem.

\begin{cor1}
For each partition $\lambda$ of $n$, one has the following decomposition of $\mathscr{P}^n_{\lambda}$:
\begin{equation*}
\left[\mathscr{P}^n_{\lambda}\right]_{\mathfrak{S}_n\times \mu_{\ell}}=\sum_{i=0}^{\ell \smin 1}{\sum_{j=0}^{\ell \smin 1}{\left[\mathrm{Ind}_{W_{\ell,n}^{\mathrm{g}_{\ell}}}^{\mathfrak{S}_n}\left((\mathscr{P}^{\mathrm{g}_{\ell}}_{\gamma_{\ell}})^{\ell}_j\boxtimes\theta_{\ell}^{i-j}\right)\boxtimes \tau_{\ell}^i\right]_{\mathfrak{S}_n\times \mu_{\ell}}}}.
\end{equation*}
\end{cor1}
\subsection{Notation}
In this last subsection of the introduction we fix notation. Let $G$ be a finite group. Denote by $\mathcal{R}(G)$ the Grothendieck ring of the category of finite-dimensional $\mathbb{C} G$-modules and $\mathcal{R}^{\mathrm{gr}}(G)$ the Grothendieck ring of the category of $\mathbb{Z}$-graded finite-dimensional $\mathbb{C} G$-modules. For $V$ a given $\mathbb{C} G$-module (resp. graded $\mathbb{C} G$-module), let $[V]_{G}$ (resp. $[V]^{\mathrm{gr}}_G$), or just $[V]$ (resp. $[V]^{\mathrm{gr}}$) denote the element in $\mathcal{R}(G)$ (resp. $\mathcal{R}^{\mathrm{gr}}(G)$ ) associated with $V$.

All schemes will be over $\mathbb{C}$ and we will also suppose that the structure morphism is separated and of finite type over $\mathbb{C}$. An algebraic variety will be an integral scheme. If $S$ is a scheme and $s \in S$, denote by $\kappa_S(s)$ the residue field of the local ring $\mathcal{O}_{S,s}$. Fix an integer $n \geq 1$, and $\Gamma$ a finite subgroup of $\mathrm{SL}_2(\mathbb{C})$. We denote the order of $\Gamma$ by $\ell$.

By a $\Gamma$-module, one means a finite-dimensional $\mathbb{C}[\Gamma]$-module. Let $\mathrm{Irr}_{\Gamma}$ be the set of all characters of irreducible representations of $\Gamma$. It is finite since $\Gamma$ is finite. Denote by $\chi_0 \in \mathrm{Irr}_{\Gamma}$ the trivial character. Moreover, the group $\Gamma$ being a subgroup of $\mathrm{SL}_2(\mathbb{C})$, it has a natural $2$-dimensional representation called the standard representation and denoted $\rho_{\mathrm{std}}$. It is irreducible whenever $\Gamma$ is not a cyclic group. In the following, the character of the standard representation will be denoted $\chi_{\mathrm{std}}$.
\\

The article is organized as follows. In the first section, we recall results obtained in \cite{Pae24} concerning the irreducible components of $\mathcal{H}_n^{\Gamma}$. In the second section, we state and prove the main result (c.f. Theorem \ref{reduction_thm}). In the third section, we dive into the combinatorial consequences of the reduction theorem when $\Gamma$ is a cyclic group and prove Corollary \ref{cor1_a}. Moreover, we prove Corollary \ref{cor1_a} in two edge cases using only representation theory and symmetric function theory, in particular avoiding Haiman's results on the isospectral Hilbert scheme. In the last section, we study the combinatorics arising from Theorem \ref{reduction_thm} when $\Gamma$ is the binary dihedral group.

\subsection{Acknowledgements}

The first author was supported in part by Research Project Grant RPG-2021-149 from The Leverhulme Trust and EPSRC grants EP-W013053-1 and EP-R034826-1. The authors would like to thank C\'edric Bonnaf\'e for suggesting we consider the Procesi bundle on fixed point components and for many fruitful discussions.

\section{Root systems and irreducible components of $\mathcal{H}_n^{\Gamma}$}

In this section, we introduce notation on root systems and recall the parameterization of irreducible components of $\mathcal{H}_n^{\Gamma}$. These are also the connected components since $\mathcal{H}_n^{\Gamma}$ is smooth thanks to the fact that $\Gamma$ is a finite group.

\begin{deff}
Define the McKay undirected multigraph  $G_{\Gamma}$ associated with $\Gamma$ the following way. The set of vertices is $I_{\Gamma} := \mathrm{Irr}_{\Gamma}$ and there is an edge between a pair of irreducible characters $(\chi,\chi')$ if and only if $\langle \chi\chi_{\mathrm{std}}|\chi'\rangle \neq 0$, with multiplicity $\langle \chi\chi_{\mathrm{std}}|\chi'\rangle$. Let $A_{\Gamma}$ denote the adjacency matrix of $G_{\Gamma}$.
\end{deff}

\begin{rmq}
Note that $G_{\Gamma}$ is indeed undirected because $\Gamma$ is a subgroup of $\mathrm{SL}_2(\mathbb{C})$.
\end{rmq}

\noindent Thanks to the McKay correspondence, one can associate to $\Gamma$ a realization $(\mathfrak{h},\Pi,\Pi^{\vee})$ of the generalized Cartan matrix $2 \mathrm{Id} - A_{\Gamma}$. Let $\tilde{W}$ denote the Weyl group associated with $(\mathfrak{h},\Pi,\Pi^{\vee})$. For each $\chi \in I_{\Gamma}$, the simple root, respectively coroot, associated with $\chi$ is denoted by  $\alpha_{\chi} \in \Pi$, respectively $\alpha_{\chi}^{\vee} \in \Pi^{\vee}$. For each $\chi \in I_{\Gamma}$, let $\Lambda_{\chi}\in \mathfrak{h}^*$ (resp. $\Lambda^{\vee}_{\chi} \in \mathfrak{h}$) be the fundamental weight (resp. fundamental coweight) associated with $\alpha^{\vee}_{\chi}$ (resp. $\alpha_{\chi}$). Let $\tilde{Q}$ (resp. $\tilde{Q}^{\vee}$) denote the root (resp. coroot) lattice of $(\mathfrak{h},\Pi,\Pi^{\vee})$. Write
\[
{\delta^{\Gamma}:=\sum_{\chi \in I_{\Gamma}}{\mathrm{dim}(X_{\chi})\alpha_{\chi}} \in \tilde{Q}}, \quad \delta^{\vee}_{\Gamma}:=\sum_{\chi \in I_{\Gamma}}{\mathrm{dim}(X_{\chi})\alpha^{\vee}_{\chi}} \in \tilde{Q}^{\vee},
\]
for the minimal positive imaginary root, resp. minimal positive imaginary coroot.
Denote by $\tilde{Q}^+ \subset \tilde{Q}$ the monoid  generated by $\Pi$.
For $d \in \tilde{Q}$ and $\chi \in I_{\Gamma}$, write $d_{\chi}$ and $|d|_{\Gamma}$ for the integers  $\langle d, \Lambda_{\chi}^{\vee} \rangle$ and $\sum_{\chi \in I_{\Gamma}}{d_{\chi}\delta^{\Gamma}_{\chi}}$ respectively.

\begin{deff}
\label{def_action}
We define a $\tilde{W}$-action on $\tilde{Q}$. Denote by $s_{\chi} \in \tilde{W}$, for $\chi \in I_{\Gamma}$, the standard generators of $\tilde{W}$ and choose $d \in \tilde{Q}$:
\begin{equation*}
(s_\chi.d)_{\xi} :=
\begin{cases}
(\sum_{h\in \overline{E_{\Gamma}}, h'=\chi}{d_{h''}}) - d_{\chi} &\text{ if } \chi=\xi\neq \chi_0 \\
(\sum_{h\in \overline{E_{\Gamma}}, h'=\chi}{d_{h''}}) - d_{\chi} + 1 &\text{ if } \chi=\xi= \chi_0 \\
d_{\xi} &\text{ else}.
\end{cases}
\end{equation*}
\end{deff}
\begin{rmq}
This action corresponds to the one defined in \cite[Definition 2.3]{Nak03} in the special case of double, one vertex framed quivers, and it is linked to the natural action by reflections on $\mathfrak{h}^*$ (denoted $*$) in the following way. Thanks to the remark at the end of \cite[Definition $2.3$]{Nak03}, one has
\begin{equation}
\label{link_eq}
\omega*(\Lambda_0-\alpha) = \Lambda_0 - \omega.\alpha, \quad \forall (\omega,\alpha) \in \tilde{W} \times \tilde{Q}
\end{equation}
where $\Lambda_0$ denotes $\Lambda_{\chi_0}$.
\end{rmq}

\noindent Let $Q$ (respectively $W$) denote the sublattice of $\tilde{Q}$ (respectively the subgroup of $\tilde{W}$) generated by ${\{\alpha_{\chi} | \chi \in I_{\Gamma}\setminus \{\chi_0\}\}}$ (respectively by $\big\{s_{\chi} | \chi \in I_{\Gamma}\setminus \{\chi_0\}\big\}$). For $a\in Q$, denote by $t_a \in \tilde{W}$ the image of $a$ under the isomorphism ${W \ltimes Q \xrightarrow{\,\smash{\raisebox{-0.65ex}{\ensuremath{\scriptstyle\sim}}}\,} \tilde{W}}$.

\begin{lemme}
\label{trans}
For each $(a,d) \in Q \times \tilde{Q}$, there exists $k \in \mathbb{Z}$ such that $t_a.d= d-a + k\delta^{\Gamma}$.
\end{lemme}
\begin{proof}
Thanks to relation (\ref{link_eq}) and \cite[Formula $6.5.2$]{kac90}, there exists $k \in \mathbb{Z}$ such that
\begin{equation*}
t_a.d =  d- a + \langle d, \delta^{\vee}_{\Gamma} \rangle + k \delta^{\Gamma}
\end{equation*}
Since $d \in \tilde{Q}$, $\langle d,\delta^{\vee}_{\Gamma}  \rangle=0$ by definition of $\delta^{\vee}_{\Gamma}$.
\end{proof}

\begin{lemme}
\label{lemme_ld}
For each $d \in \tilde{Q}$, there exists a unique integer $r$ such that $d$ and $r\delta^{\Gamma}$ are in the same $\tilde{W}$-orbit for the $.$ action of Definition \ref{def_action}.
\end{lemme}
\begin{proof}
Take $d \in \tilde{Q}$. Then $a:=d-d_0\delta^{\Gamma} \in Q$ and thanks to Lemma \ref{trans}, $t_a.d$ is an element of the desired form. Now suppose  that there are two integers  $r_1$ and $r_2$ such that $r_1\delta^{\Gamma}$ and $r_2\delta^{\Gamma}$ are in the same $\tilde{W}$-orbit. Since $\delta^{\Gamma}$ is in the kernel of the generalized Cartan matrix $2\mathrm{Id}-A_{\Gamma}$, $\delta^{\Gamma}$ is fixed under the action of $W$-action. This observation reduces the $\tilde{W}$-orbit of $r_1\delta^{\Gamma}$ to the $Q$-orbit. There must then exist $a \in Q$ such that ${t_a.r_1\delta^{\Gamma}=r_2\delta^{\Gamma}}$. Using Lemma \ref{trans}, there exists $k \in \mathbb{Z}$ such that $t_a.r_1\delta^{\Gamma}= r_1\delta^{\Gamma}-a\text{ } + k\delta^{\Gamma}$, one can conclude that $a=0$ and that $r_1=r_2$.
\end{proof}

\begin{deff}
\label{deff_weight}
The \textit{weight} $\mathrm{wt}(d)$ of $d \in \tilde{Q}$ is the unique integer $r$ such that $r\delta^{\Gamma}$ and $d$ are in the same $\tilde{W}$-orbit.
\end{deff}

\noindent Now that the weight of an element of $\tilde{Q}$ has been defined, we can recall the parameterization of connected components of the fixed point locus. Set
\[
{\mathcal{A}^n_{\Gamma}:=\left\{\alpha \in \tilde{Q}^+ \big | |\alpha|_{\Gamma}=n \text{ and } \mathrm{wt}(\alpha) \geq 0\right\}}.
\]
It is shown in \cite[Corollary $3.3$]{Pae24} that the set $\mathcal{A}^n_{\Gamma}$ indexes the irreducible components of $\mathcal{H}_n^{\Gamma}$. Indeed, if one denotes by ${{H}_n^{\Gamma,d}:=\left\{I \in \mathcal{H}_n^{\Gamma} | \mathrm{Tr}(\mathbb{C}[x,y]/I)=\sum_{\chi \in I_{\Gamma}}{d_{\chi}\chi}\right\}}$, then
\begin{equation*}
\mathcal{H}_n^{\Gamma}= \coprod_{d \in \mathcal{A}^n_{\Gamma}}{\mathcal{H}_n^{\Gamma,d}}.
\end{equation*}
By \cite[Proposition $3.11$]{Pae24}, the connected component $\mathcal{H}_n^{\Gamma,d}$ labelled by $d \in \mathcal{A}^n_{\Gamma}$ has dimension $2\mathrm{wt}(d)$. The restriction of $\mathscr{P}^n$ to each connected component $\mathcal{H}_n^{\Gamma,d}$ of $\mathcal{H}_n^{\Gamma}$ defines a vector bundle and the fibers of this vector bundle are $(\mathfrak{S}_n\times \Gamma)$-modules.

\section{The Reduction Theorem}
\label{sec_reduction}
In this section, we state and  prove the main result of the article. We begin with some preliminary results. Fix $d \in \mathcal{A}^n_{\Gamma}$. To improve readability, we set $r_d:=\mathrm{wt}(d)$. Consider ${d_0=d-r_d\delta^{\Gamma}}$  and note that, by construction, $\mathrm{wt}(d_0)=0$. We fix ${\mathrm{g}_{\Gamma}:=|d_0|_{\Gamma}}$. The connected component $\mathcal{H}_{\mathrm{g}_{\Gamma}}^{\Gamma,{d_0}} \subset \mathcal{H}_{\mathrm{g}_{\Gamma}}^{\Gamma}$ is zero-dimensional and we write $I_{d_0}$ for the unique ideal of $\mathbb{C}[x,y]$ belonging to $\mathcal{H}_{\mathrm{g}_{\Gamma}}^{\Gamma,{d_0}}$.

\begin{lemme}
\label{lemme_supp_0}
The image of $I_{d_0}$ under $\sigma_{\mathrm{g}_{\Gamma}}$ is the point $\overline{0}\in {(\mathbb{C}^2)}^{\mathrm{g}_{\Gamma}}/\mathfrak{S}_{\mathrm{g}_{\Gamma}}$.
\end{lemme}
\begin{proof}
Consider the diagonal $\mathbb{C}^{\times}$-action on $\mathbb{C}^2$ given by
\begin{equation*}
{\forall (t,(x,y)) \in \mathbb{C}^{\times} \times \mathbb{C}^2, t.(x,y):=(tx,ty)}.
\end{equation*}
This action induces a $\mathbb{C}^{\times}$-action on $\mathcal{H}_{\mathrm{g}_{\Gamma}}$ which commutes with the $\Gamma$-action and the Hilbert-Chow morphism $\sigma_{\mathrm{g}_{\Gamma}}$ is $\mathbb{C}^{\times}$-equivariant. The fact that $\mathbb{C}^{\times}$ is connected and the irreducible component $\mathcal{H}_{\mathrm{g}_{\Gamma}}^{\Gamma,{d_0}}$ equals $\{I_{d_0}\}$ implies that $I_{d_0}$ is a $\mathbb{C}^{\times}$-fixed point. This ideal must then be mapped by $\sigma_{\mathrm{g}_{\Gamma}}$ to a $\mathbb{C}^{\times}$-fixed point of ${(\mathbb{C}^2)}^{\mathrm{g}_{\Gamma}}/\mathfrak{S}_{\mathrm{g}_{\Gamma}}$. Finally, we note that $\overline{0}$ is the only fixed point in ${(\mathbb{C}^2)}^{\mathrm{g}_{\Gamma}}/\mathfrak{S}_{\mathrm{g}_{\Gamma}}$.
\end{proof}

Denote by $U^f$ the following open subset of ${(\mathbb{C}^2)}^{r_d}$:
\begin{equation*}
\{(p_1,\dots ,p_{r_d}) \in {(\mathbb{C}^2\setminus \{(0,0)\})}^{r_d}\big{|} \forall (i,j) \in \llbracket 1, r_d \rrbracket^2, i \neq j \Rightarrow \Gamma p_i \cap \Gamma p_j = \emptyset\}.
\end{equation*}
Let $D_{d_0}:=\{I_{d_0} \cap  \bigcap_{j=1}^{r_d}{I(\Gamma p_j)} \subset \mathbb{C}[x,y] \big{|} (p_1,\dots, p_{r_d}) \in U^f\}$.

\begin{lemme}
The set $D_{d_0}$ is a dense open subset of $\mathcal{H}_n^{\Gamma,d}$.
\end{lemme}
\begin{proof}
In type $A$, this is \cite[Lemma $7.8.(i)$]{Gor08}.
Take $I=I_{d_0} \cap  \bigcap_{j=1}^{r_d}{I(\Gamma p_j)} \in D_{d_0}$. Lemma \ref{lemme_supp_0} implies that ${V(I_{d_0})=\overline{0}}$. Therefore, for all $j \in \llbracket 1, r_d \rrbracket$ we have ${V(I_{d_0}) \cap \Gamma p_i = \emptyset}$, which gives an isomorphism of $\Gamma$-modules
\begin{equation*}
\label{splitting_iso}
\mathbb{C}[x,y]/I \simeq \mathbb{C}[x,y]/I_{d_0} \oplus \bigoplus_{j=1}^{r_d}{\mathbb{C}[x,y]/I(\Gamma p_j)}.
\end{equation*}
This isomorphism shows that $I$ is of codimension $\mathrm{g}_{\Gamma} + r_d\ell=n$ in $\mathbb{C}[x,y]$ and that the character of the  $\Gamma$-module $\mathbb{C}[x,y]/I$ is $d$. This means that $D_{d_0} \subset \mathcal{H}_n^{\Gamma,d}$. The association $(p_1,\dots ,p_{r_d}) \mapsto \mathbb{C}[x,y]/I$ defines a vector bundle over $U^f$ whose fibers are cyclic $\mathbb{C}[x,y]$-modules of dimension $n$. Thus, there is a (unique) morphism $U^f \to \mathcal{H}_n$ such that this vector bundle is the pull-back of the tautological bundle on $\mathcal{H}_n$. The fibres of this morphism are finite i.e. it is a quasi-finite morphism. Hence, by Zariski's Main Theorem \cite[Théorème $8.12.6$]{EGA43}, the image $D_{d_0}$ of $U^f$ is a (connected) locally closed subset of $\mathcal{H}_n$ of dimension $2 r_d$. Since $D_{d_0}$ is contained in $\mathcal{H}_n^{\Gamma,d}$ and the latter also has dimension $2r_d$, we deduce that $D_{d_0}$ is an open dense subset of $\mathcal{H}_n^{\Gamma,d}$.
\end{proof}

\noindent Throughout the remainder of this section we fix $(p_1,\dots,p_{r_d}) \in U^f$. Denote by $J$ the ideal $\bigcap_{j=1}^{r_d}{I(\Gamma p_j)}$ and set ${I_d:=I_{d_0} \cap J \in \mathcal{H}_n^{\Gamma,d}}$. Let ${q:=(\Gamma p_1,\dots,\Gamma p_{r_d}) \in {(\mathbb{C}^2)}^{r_d\ell}}$ and ${p:=(0,q)}$, which is a point in ${\pi_n^{-1}(\sigma_n(I_d))\subset {(\mathbb{C}^2)}^n}$. Denote by $S_p$ the stabilizer of $p$ in $\mathfrak{S}_n \times \Gamma$.
\begin{rmq}
By construction,  $J$ is an element of $\mathcal{H}_{r_d\ell}^{\Gamma}$.
\end{rmq}

\noindent By our choice of $p \in (\mathbb{C}^2)^n$, $S_p$ is a subgroup of the product $\mathfrak{S}_{\mathrm{g}_{\Gamma}} \times \mathfrak{S}_{r_d \ell} \times \Gamma$. Let ${\nabla\colon \Gamma \to S_p}$ be the morphism of groups making the following diagram of groups morphisms commute
\begin{center}
\begin{tikzcd}
&\mathfrak{S}_{\mathrm{g}_{\Gamma}}& \\
*\ar[ur] & S_p \ar[r, hook] \ar[u] & (\mathfrak{S}_{\mathrm{g}_{\Gamma}} \times \mathfrak{S}_{r_d\ell} \times \Gamma),  \ar[dl, "p_3"] \ar[ul, "p_1"']\\
& \Gamma \ar[u, "\nabla"] \ar[ul]&
\end{tikzcd}
\end{center}
where $*$ is the trivial group. Note that this diagram implies that $\nabla(\Gamma)$ is a subgroup of $(\mathfrak{S}_{r_d\ell} \times \Gamma)$. Then we have a group isomorphism
\[
\textphnc{\As} \colon \begin{array}{ccc}
S_p & \xrightarrow{\,\smash{\raisebox{-0.65ex}{\ensuremath{\scriptstyle\sim}}}\,} & \mathfrak{S}_{\mathrm{g}_{\Gamma}} \times \Gamma\\
(x_1,x_2,\gamma) & \mapsto & (x_1,\gamma)
\end{array}
\]
with inverse given by $(\sigma,\gamma) \mapsto \sigma\nabla(\gamma)$.

\noindent We can now state the main theorem of this article.
\begin{thm}
\label{reduction_thm}
The isomorphism $\textphnc{\As}$, endows $\mathscr{P}^{\mathrm{g}_{\Gamma}}_{|I_{d_0}}$ with a $S_p$-module structure, such that for each $I \in \mathcal{H}_n^{\Gamma,d}$,
\begin{equation*}
\left[\mathscr{P}^n_{|I}\right]_{\mathfrak{S}_n \times \Gamma}=\left[\mathrm{Ind}_{S_p}^{\mathfrak{S}_n\times \Gamma}\left(\mathscr{P}^{\mathrm{g}_{\Gamma}}_{|I_{d_0}}\right)\right]_{\mathfrak{S}_n \times \Gamma}.
\end{equation*}
\end{thm}
\noindent The proof is postponed to the end of the section since we first require several intermediate results. To improve readability, set
\begin{itemize}
\item $x^0:=(I_{d_0},0) \in \mathcal{X}_{\mathrm{g}_{\Gamma}}$
\item ${x^q:=(J,q) \in \mathcal{X}_{r_d\ell}}$
\item ${x^p:=(I_d,p) \in \mathcal{X}_n}$
\item $x^{(0,q)}:=((I_{d_0},0),(J,q)) \in \mathcal{X}_{\mathrm{g}_{\Gamma}}\times \mathcal{X}_{r_d\ell}$.
\end{itemize}
We fix affine open subsets $U^{d} \subset \mathcal{H}_n$ and $U^{d^0} \subset \mathcal{H}_{\mathrm{g}_{\Gamma}}$, containing $I_d$ and $I_{d_0}$ respectively. Since $S_p$ is a finite group and $I_d$ is fixed by $S_p$, we may assume that $U^d$ is $S_p$-stable. Set ${A^{d}:= \mathcal{O}_{\mathcal{H}_n}(U^d)}$ and  $A^{d_0}:=\mathcal{O}_{\mathcal{H}_{\mathrm{g}_{\Gamma}}}(U^{d_0})$. The morphisms $\rho_n$ and $\rho_{\mathrm{g}_{\Gamma}}$ are finite and thus affine. Therefore, we also fix $B^d:=\mathcal{O}_{\mathcal{X}_n}(\rho_n^{-1}(U^d))$ and ${B^{d_0}:=\mathcal{O}_{\mathcal{X}_{\mathrm{g}_{\Gamma}}}(\rho_{\mathrm{g}_{\Gamma}}^{-1}(U^{d_0}))}$.

If $R$ is a commutative ring and $M$ an $R$-module then we write $\mathrm{Ann}_{R}(M)$ for its annihilator $\{r \in R | \forall m \in M, r.m=0\}$ and $\mathrm{Supp}_{R}(M)=\{p \in \mathrm{Spec}(R)| M_{p} \neq 0\}$ for the support of $M$.
\begin{lemme}
\label{lemme_ind_fib}
Let $G$ be a finite group acting on an affine variety $V$ over $\mathbb{C}$. Let $M$ be a finite dimensional $\mathbb{C}[V] \rtimes G$-module such that $\mathrm{Supp}_{\mathbb{C}[V]} (M)$  is a $G$-orbit. Let $q \in \mathrm{Supp}_{\mathbb{C}[V]}(M)$ and denote by $G_q$ the stabilizer of $q$ in $G$. Then there is an isomorphism of $(\mathbb{C}[V] \rtimes G)$-modules
\begin{equation*}
M \simeq \mathrm{Ind}_{\mathbb{C}[V] \rtimes G_q}^{\mathbb{C}[V] \rtimes G}\left(M_q\right).
\end{equation*}
\end{lemme}
\begin{proof}
Since the module $M$ is finite-dimensional, \cite[Theorem $2.13$]{Eis95} says that the diagonal map $\tilde{\phi} \colon M \to \bigoplus_{p \in \mathrm{Supp}_{\mathbb{C}[V]}(M)}{M_p}$ is an isomorphism of $\mathbb{C}[V]$-modules. For each  $g \in G$, multiplication $m \mapsto g . m$ defines an isomorphism of $\mathbb{C}$-vector spaces ${M_q \xrightarrow{\,\smash{\raisebox{-0.65ex}{\ensuremath{\scriptstyle\sim}}}\,} M_{g.q}}$. Therefore, we may rewrite $\tilde{\phi}$ as
\begin{equation*}
\phi\colon M \xrightarrow{\,\smash{\raisebox{-0.65ex}{\ensuremath{\scriptstyle\sim}}}\,} \bigoplus_{\bar{g} \in G/G_q}{M_{\bar{g}.q}} \text{  }.
\end{equation*}
Again using \cite[Theorem $2.13$]{Eis95}, we idenitify $M_q$ with the subspace of $M$ consisting of sections killed by some power of the maximal ideal $m_q$ defining $q$. Then there is a canonical multiplication map
\[
\psi \colon (\mathbb{C}[V] \rtimes G) \otimes_{\mathbb{C}[V] \rtimes G_q} M_q \to M
\]
of $(\mathbb{C}[V]\rtimes G)$-modules given by $f_g g \otimes m \mapsto f_g . m$. Since the composite
\begin{center}
$\phi \circ \psi\colon \begin{array}{ccc}
(\mathbb{C}[V] \rtimes G) \otimes_{\mathbb{C}[V] \rtimes G_q} M_q & \to & \bigoplus_{\bar{g} \in G/G_q}{M_{\bar{g}.q}}\\
f_gg \otimes m & \mapsto & f_gg.m \\
\end{array}$\\
\end{center}
is an isomorphism, $\psi$ must also be an isomorphism.
\end{proof}

\begin{lemme}
\label{lemme_etale_ss}
Let $R$ and $S$ be local noetherian $\mathbb{C}$-algebras. If $f\colon R\to S$ is an unramified morphism of local rings and $M$ is an $S$-module that is $R$-semisimple then $M$ is $S$-semisimple.
\end{lemme}
\begin{proof}
Let $\chi\colon R \to \mathbb{C}$ be the algebra morphism defined by the maximal ideal $m_R$ of $R$. The module $M$ being $R$-semisimple means that $M={\{m \in M| \forall r \in R, r.m=\chi(r)m\}}$. The action of $S$ on $M$ factors though $S/m_{R}S$. Since the morphism $f$ is unramified, the quotient $S/m_{R}S$ equals the residue field $S/m_{S}S$ of $S$; see e.g. \cite[\href{https://stacks.math.columbia.edu/tag/02GF}{Tag 02GF}]{sp20}. The ring $S/m_{R}S$ is thus a semisimple ring, which implies that $M$ is $S$-semisimple.
\end{proof}

\noindent Write ${{{(\mathbb{C}^2)}^{r_d\ell}}^{\circ}:=\left\{(x_1,...,x_{r_d\ell}) \in {(\mathbb{C}^2)}^{r_d\ell} | \forall i \neq j \Rightarrow x_i \neq x_j\right\}}$ for the complement to the big diagonal. Restricting \eqref{eq:cartprod} to ${{(\mathbb{C}^2)}^{r_d\ell}}^{\circ}$ gives a commutative diagram:
\begin{center}
\begin{tikzcd}
\mathcal{X}^{\circ}_{r_d\ell}:=f_{r_d\ell}^{-1}\left({{(\mathbb{C}^2)}^{r_d\ell}}^{\circ}\right) \ar[r,"\sim","f^{\circ}_{r_d\ell}"'] \ar[d, "\rho^{\circ}_{r_d\ell}"'] & {{(\mathbb{C}^2)}^{r_d\ell}}^{\circ} \ar[d,"\pi^{\circ}_{r_d\ell}"]\\
\mathcal{H}_{r_d\ell}^{\circ}:=\sigma_{r_d\ell}^{-1}\left({{(\mathbb{C}^2)}^{r_d\ell}}^{\circ}\right) \ar[r,"\sim", "\sigma_{r_d\ell}^{\circ}"'] & {{(\mathbb{C}^2)}^{r_d\ell}}^{\circ} / \mathfrak{S}_{r_d\ell}.
\end{tikzcd}
\end{center}
The Hilbert-Chow morphism $\sigma_{r_d\ell}$ is a crepant resolution of singularities, which is an isomorphism over the smooth locus of ${{(\mathbb{C}^2)}^{r_d\ell}} / \mathfrak{S}_{r_d\ell}$. Therefore, the morphism $\sigma_{r_d\ell}^{\circ}$ is an isomorphism. This implies that $f^{\circ}_{r_d\ell}$ is also an isomorphism.
Consider now the morphism
\[
\tilde{h} \colon {(\mathbb{C}^2)}^{r_d\ell} \to {(\mathbb{C}^2)}^n/\mathfrak{S}_n
\]
sending $x$ to the orbit $\overline{(0,x)}$ of $(0,x) \in {(\mathbb{C}^2)}^n$.
%\begin{equation*}
%\tilde{h}:\begin{array}{ccc}
%{(\mathbb{C}^2)}^{r_d\ell}& \to& {(\mathbb{C}^2)}^n/\mathfrak{S}_n\\
%x & \mapsto & \overline{(0,x)}\\
%\end{array}.
%\end{equation*}
The morphism $\tilde{h}$ is finite since it is the composition of the finite morphism ${{(\mathbb{C}^2)}^{r_d\ell} \to {(\mathbb{C}^2)}^n}$ with the (finite) quotient morphism ${{(\mathbb{C}^2)}^n \to {(\mathbb{C}^2)}^n/\mathfrak{S}_n}$. In particular, $\mathrm{Im}(\tilde{h})$ is a closed subscheme of ${(\mathbb{C}^2)}^n/\mathfrak{S}_n$. One can then consider ${h\colon {(\mathbb{C}^2)}^{r_d\ell} \to \mathrm{Im}(\tilde{h})}$.
\begin{lemme}
\label{lemme_C2r_et}
The morphism $h$ is étale when restricted to ${{(\mathbb{C}^2)}^{r_d\ell}}^{\circ}$ .
\end{lemme}
\begin{proof}
The fact that the morphism $\tilde{h}$ is finite implies that $h$ is finite. Therefore, it is enough to prove that $h$ is smooth over ${{(\mathbb{C}^2)}^{r_d\ell}}^{\circ}$.

Set ${\mathbb{Z}^+_n:=\{(k_1,k_2) \in {(\mathbb{Z}_{\geq 0})}^2| 0 \leq k_1+k_2 \leq n\}}$. For $(k_1,k_2) \in \mathbb{Z}_n^+$, let
\begin{equation*}
f_{k_1,k_2}(X_1,...,X_n,Y_1,...,Y_n):= \sum_{i=1}^{n}{X_i^{k_1}Y_i^{k_2}} \in \mathbb{C}[{(\mathbb{C}^2)}^n]^{\mathfrak{S}_n},
\end{equation*}
\begin{equation*}
P_{k_1,k_2}:= \sum_{i=\mathrm{g}_{\Gamma}+1}^n{Z_i^{k_1}T_i^{k_2}} \in \mathbb{C}[Z_{\mathrm{g}_{\Gamma}+1},...,Z_n,T_{\mathrm{g}_{\Gamma}+1},...,T_n].
\end{equation*}
Thanks to \cite[Chapter II, section $3$]{Weyl66}, the set $\{f_{k_1,k_2} | (k_1,k_2) \in \mathbb{Z}^+_n \}$ is a set of generators of $\mathbb{C}[{(\mathbb{C}^2)}^n]^{\mathfrak{S}_n}$. Moreover, the set $\{P_{k_1,k_2} | (k_1,k_2) \in \mathbb{Z}^+_n\}$ is a set of generators of $\mathbb{C}[{(\mathbb{C}^2)}^{r_d\ell}]^{\mathfrak{S}_{r_d\ell}}$. By definition,
\begin{equation*}
\tilde{h}^{\sharp}\colon \begin{array}{ccc}
\mathbb{C}[{(\mathbb{C}^2)}^n]^{\mathfrak{S}_n} & \to & \mathbb{C}[{(\mathbb{C}^2)}^{r_d\ell}]\\
f_{k_1,k_2} & \mapsto & P_{k_1,k_2} \\
\end{array}
\end{equation*}
This implies that $\mathbb{C}[\mathrm{Im}(\tilde{h})]=\mathbb{C}[Z_{\mathrm{g}_{\Gamma}+1},...,Z_n,T_{\mathrm{g}_{\Gamma}+1},...,T_n]^{\mathfrak{S}_{r_d\ell}}$ and in particular that
\begin{equation*}
{\mathrm{Im}(\tilde{h}) \simeq {(\mathbb{C}^2)}^{r_d\ell}/\mathfrak{S}_{r_d\ell} \subset {(\mathbb{C}^2)}^n/\mathfrak{S}_n}.
\end{equation*}
This allows us to identify $h$ with the morphism ${h\colon {(\mathbb{C}^2)}^{r_d\ell} \to {(\mathbb{C}^2)}^{r_d\ell}/\mathfrak{S}_{r_d\ell}}$. It is then clear that the restriction of $h$ to ${{(\mathbb{C}^2)}^{r_d\ell}}^{\circ}$ is smooth. Indeed, $h$ is finite and the $\mathfrak{S}_{r_d\ell}$-action on ${{(\mathbb{C}^2)}^{r_d\ell}}^{\circ}$ is free, which implies that ${{(\mathbb{C}^2)}^{r_d\ell}}^{\circ}/\mathfrak{S}_{r_d\ell}$ is smooth.

\end{proof}

\noindent The stalk of $\mathscr{P}^n$ at $I \in \mathcal{H}_n$ is denoted $\mathscr{P}^n_I$. The isospectral Hilbert scheme $\mathcal{X}_n$ is an algebraic variety over $\mathcal{H}_n\times {(\mathbb{C}^2)}^n$. This implies that the fiber ${\mathscr{P}^n_{|I_d}:=\mathscr{P}^n_{I_d} \otimes_{\mathcal{O}_{\mathcal{H},I_d}} \kappa_{\mathcal{H}_n}(I_d)}$ of the Procesi bundle is a $\mathbb{C}[{(\mathbb{C}^2)}^n]$-module. It is moreover an $\mathfrak{S}_n \times \Gamma$-module. This endows $\mathscr{P}^n_{|I_d}$ with a structure of ${\big (\mathbb{C}[{(\mathbb{C}^2)}^n] \rtimes (\mathfrak{S}_n \times \Gamma)\big)}$-module. Recall that $p$ is the point $(0,q)$ of $(\mathbb{C}^2)^n$.

\begin{lemme}\label{lem:surjlocalProcesi}
There exists a surjective morphism of rings $\textphnc{q}\colon \mathcal{O}_{\mathcal{X}_n,x^p} \twoheadrightarrow {(\mathscr{P}^n_{|I_d})}_p$.
\end{lemme}
\begin{proof}
Let us construct $\textphnc{q}$ locally around $I_d$. Let $m_{I_d} \in \mathrm{Spec}(A^d)$ be the maximal ideal of $A^d$ corresponding to $I_d$ and $m_{x^p} \in \mathrm{Spec}(B^d)$ be the maximal ideal of $B^d$ corresponding to $x^p$. By definition, the stalk $\mathscr{P}^n_{I_d}$ equals $B^d\otimes_{A^d}A^d_{m_{I_d}}$. Moreover, the fiber of the associated vector bundle $\mathscr{P}^n_{|I_d}$ is isomorphic to $\mathscr{P}^n_{I_d} \otimes_{A^d_{m_{I_d}}}\left(A^d_{m_{I_d}}/m_{I_d}A^d_{m_{I_d}}\right)$, which is then isomorphic to $B^d/m_{I_d}B^d$. The localization of $(\mathscr{P}^n_{|I_d})$ at the maximal ideal associated with $p$ in  $\mathbb{C}[{(\mathbb{C}^2)}^n]$ is isomorphic to ${B^d/m_{{I_d}}B^d \otimes_{B^d} B^d_{x^p} \simeq B^d_{x^p}/m_{I_d}B^d_{x^p}}$. Finally, one has $\mathcal{O}_{\mathcal{X}_n,x^p} \simeq B^d_{x^p}$, which makes the construction of the desired morphism canonical. Indeed, it is just the quotient map $B^d_{x^p} \twoheadrightarrow B^d_{x^p}/m_{I_d}B^d_{x^p}$.\\
\end{proof}

\noindent Let us denote by  $V$ the following open set of ${(\mathbb{C}^2)}^n$
\begin{equation*}
\{(s_1,...,s_{\mathrm{g}_{\Gamma}},\Gamma t_1,...,\Gamma t_{r_d}) \in {(\mathbb{C}^2)}^n \big{|} \forall (i,j,\gamma) \in \llbracket 1,\mathrm{g}_{\Gamma}\rrbracket \times \llbracket 1, r_d\rrbracket \times \Gamma, s_i \neq \gamma.t_j\}.
\end{equation*}
Applying the key factorization result \cite[Lemma $3.3.1$]{H01}, one has
\begin{equation}\label{eq:betamap}
\beta\colon \begin{array}{ccc}
(f_{\mathrm{g}_{\Gamma}} \times f_{r_d\ell})^{-1}(V) & \xrightarrow{\,\smash{\raisebox{-0.65ex}{\ensuremath{\scriptstyle\sim}}}\,} & f_n^{-1}(V)\\
((I,u),(I',u')) & \mapsto & (I\cap I', (u,u')),\\
\end{array}
\end{equation}
which is an isomorphism of schemes over ${(\mathbb{C}^2)}^n$. Let ${\alpha\colon f_n^{-1}(V) \xrightarrow{\,\smash{\raisebox{-0.65ex}{\ensuremath{\scriptstyle\sim}}}\,} (f_{\mathrm{g}_{\Gamma}} \times f_{r_d\ell})^{-1}(V)}$ be the inverse morphism to $\beta$. By construction, $p \in V$. The isomorphism $\alpha$ induces an isomorphism of local rings
\begin{equation*}
\alpha^{\sharp}_{x^p}\colon \mathcal{O}_{\mathcal{X}_{\mathrm{g}_{\Gamma}}\times \mathcal{X}_{r_d\ell},x^{(0,q)}} \xrightarrow{\,\smash{\raisebox{-0.65ex}{\ensuremath{\scriptstyle\sim}}}\,} \mathcal{O}_{\mathcal{X}_n,x^p}.
\end{equation*}
Denote $\iota_{\mathrm{g}_{\Gamma}}\colon \mathcal{X}_{\mathrm{g}_{\Gamma}} \to \mathcal{X}_{\mathrm{g}_{\Gamma}} \times \mathcal{X}_{r_d\ell}$ the morphism that, set theoretically, maps ${(I,u) \in \mathcal{X}_{\mathrm{g}_{\Gamma}}}$ to $((I,u),(J,q)) \in \mathcal{X}_{\mathrm{g}_{\Gamma}} \times \mathcal{X}_{r_d\ell}$. The morphism  $\iota$ is a closed immersion. On the level of stalks, one has $\iota^{\sharp}_{x^0}\colon \mathcal{O}_{\mathcal{X}_{\mathrm{g}_{\Gamma}}\times \mathcal{X}_{r_d\ell},x^{(0,q)}} \twoheadrightarrow \mathcal{O}_{\mathcal{X}_{\mathrm{g}_{\Gamma}},x^0}$ and we write $K$ for the kernel. The following proposition is key to the main result since it allows us to identify the summand ${(\mathscr{P}^n_{|I_d})}_p$ of $\mathscr{P}^n_{|I_d}$ with the fiber $\mathscr{P}^{g_{\Gamma}}_{|I_{d_0}}$ of the Procesi bundle on $\mathcal{H}_{g_{\Gamma}}$.

\begin{prop}
\label{h_exists}
There exists a surjective morphism $\textphnc{\Ahd}\colon \mathcal{O}_{\mathcal{X}_{\mathrm{g}_{\Gamma}},x^0} \to {(\mathscr{P}^n_{|I_d})}_p$ such that the following diagram commutes
\begin{center}
\begin{tikzcd}
K \ar[d, hook] & & \\
\mathcal{O}_{\mathcal{X}_{\mathrm{g}_{\Gamma}}\times \mathcal{X}_{r_d\ell},x^{(0,q)}} \ar[r,"\alpha^{\sharp}_{x^p}"',"\sim"] \ar[d,"\iota^{\sharp}_{x^0}"', two heads]  & \mathcal{O}_{\mathcal{X}_n,x^p}\ar[r, two heads, "\textphnc{q}"] & {(\mathscr{P}^n_{|I_d})}_p. \\
\mathcal{O}_{\mathcal{X}_{\mathrm{g}_{\Gamma}},x^0} \ar[urr,"\textphnc{\Ahd}"'] & &
\end{tikzcd}
\end{center}
\end{prop}
\begin{proof}
It is enough to show that $\textphnc{q}(\alpha^{\sharp}_{x^p}(K))=0$.
\noindent Since the point $q$ is a collection of $r_d$-free and distinct $\Gamma$-orbits, it belongs to ${{(\mathbb{C}^2)}^{r_d\ell}}^{\circ}$. One then has the following isomorphism of local rings
\begin{equation*}
(\mathrm{id}_{\mathcal{X}_{\mathrm{g}_{\Gamma}}} \times f_{r_d\ell})^{\sharp}_{x^{(0,q)}}\colon \mathcal{O}_{\mathcal{X}_{\mathrm{g}_{\Gamma}}\times {{(\mathbb{C}^2)}^{r_d\ell}}^{\circ}, (x^0,q)} \xrightarrow{\,\smash{\raisebox{-0.65ex}{\ensuremath{\scriptstyle\sim}}}\,} \mathcal{O}_{\mathcal{X}_{\mathrm{g}_{\Gamma}}\times \mathcal{X}_{r_d\ell},x^{(0,q)}}
\end{equation*}
Note that $(\mathrm{id}_{\mathcal{X}_{\mathrm{g}_{\Gamma}}} \times f_{r_d\ell})$ is a morphism over ${(\mathbb{C}^2)}^n$.
To keep the notation concise, we denote the preceding isomorphism by $f^{\sharp}_{x^{(0,q)}}$. This new piece of information gives
\begin{equation}
\label{big_diag}
\begin{tikzcd}[column sep=normal, row sep=large]
\tilde{K}\ar[r, "\sim"] \ar[d, hook] & K \ar[d, hook] & \\
\mathcal{O}_{\mathcal{X}_{\mathrm{g}_{\Gamma}}\times {{(\mathbb{C}^2)}^{r_d\ell}}^{\circ}, (x^0,q)} \ar[d, two heads] \ar[r, "\sim", "f^{\sharp}_{x^{(0,q)}}"']& \mathcal{O}_{\mathcal{X}_{\mathrm{g}_{\Gamma}}\times \mathcal{X}_{r_d\ell},x^{(0,q)}} \ar[r,"\alpha^{\sharp}_{x^p}"',"\sim"] \ar[dl,"\iota^{\sharp}_{x^0}", two heads]  & \mathcal{O}_{\mathcal{X}_n,x^p}\ar[r, two heads, "\textphnc{q}"] & {(\mathscr{P}^n_{|I_d})}_p \\
\mathcal{O}_{\mathcal{X}_{\mathrm{g}_{\Gamma}},x^0} & & &
\end{tikzcd}
\end{equation}
Let $X_i$ and $Y_i$ for $i \in \llbracket 1,n\rrbracket$ be the coordinates functions on ${(\mathbb{C}^2)}^n$. Then
\begin{equation*}
\tilde{K}=\langle X_{\mathrm{g}_{\Gamma}+1}-X_{\mathrm{g}_{\Gamma}+1}(q),Y_{\mathrm{g}_{\Gamma}+1}-Y_{\mathrm{g}_{\Gamma}+1}(q),...,X_n-X_n(q),Y_n-Y_n(q)\rangle.
\end{equation*}
Let us denote $X_i-X_i(q)$ by $\tilde{X}_i$ and $Y_i-Y_i(q)$ by $\tilde{Y}_i$ so that the kernel $K$ is equal to
\begin{equation*}
\langle f_{r_d\ell}^{\sharp}(\tilde{X}_{\mathrm{g}_{\Gamma}+1}),f_{r_d\ell}^{\sharp}(\tilde{Y}_{\mathrm{g}_{\Gamma}+1}),...,f_{r_d\ell}^{\sharp}(\tilde{X}_n),f_{r_d\ell}^{\sharp}(\tilde{Y}_n)\rangle.
\end{equation*}
Proving that $\textphnc{q}\left(\alpha^{\sharp}_{x^p}(K)\right)=0$ amounts to showing that for all $i \in \llbracket \mathrm{g}_{\Gamma}+1,n\rrbracket$
\begin{align*}
\textphnc{q}\big(\alpha^{\sharp}_{x^p}(f_{r_d\ell}^{\sharp}(X_i))\big)&=X_i(q)\\
\textphnc{q}\big(\alpha^{\sharp}_{x^p}(f_{r_d\ell}^{\sharp}(Y_i))\big)&=Y_i(q)
\end{align*}
Let us focus on diagram (\ref{big_diag}). Zooming in on the left part gives
\begin{center}
\begin{tikzcd}
\mathcal{O}_{\mathcal{X}_{\mathrm{g}_{\Gamma}}\times {{(\mathbb{C}^2)}^{r_d\ell}}^{\circ}, (x^0,q)} \ar[d,"\sim" lablv]& \mathbb{C}[{{(\mathbb{C}^2)}^{r_d\ell}}^{\circ}]_{q}\ar[l] \\
\mathcal{O}_{\mathcal{X}_n,x^p} & \mathbb{C}[{(\mathbb{C}^2)}^n]_{p}\ar[u] \ar[l,"f^{\sharp}_{x^p}"]\\
\mathcal{O}_{\mathcal{H}_n,I_d} \ar[u, "\rho^{\sharp}_{x^p}"']& \mathbb{C}[{(\mathbb{C}^2)}^n]^{\mathfrak{S}_n}_{p}. \ar[u,"\pi^{\sharp}_{p}"] \ar[l, "\sigma^{\sharp}_{I_d}"']
\end{tikzcd}
\end{center}
The upper square commutes because the isomorphism  $f_{x^p}$ is  an isomorphism over ${(\mathbb{C}^2)}^n$. Now zooming in on the right hand side of (\ref{big_diag}) gives
\begin{center}
\begin{tikzcd}
\mathcal{O}_{\mathcal{H}_n,I_d} \ar[r, "\rho^{\sharp}_{x^p}"] \ar[dr, two heads]& \mathcal{O}_{\mathcal{X}_n,x^p} \ar[r, "\textphnc{q}",two heads]  & {(\mathscr{P}^n_{|I_d})}_{p}\\
& \kappa_{\mathcal{H}_n}(I_d) \ar[ur]&
\end{tikzcd}
\end{center}

\noindent The fact that the preceding diagram commutes is clear once one comes back to the description in terms of rings of functions on affine open subsets:
\begin{center}
\begin{tikzcd}
A^d_{m_{I_d}} \ar[r, "\rho^{\sharp}_{x^p}"] \ar[dr, two heads]& B^d_{x^p} \ar[r, "\textphnc{q}",two heads]  & B^d_{x^p}/m_{I_d}B^d_{x^p}\\
& A^d_{m_{I_d}}/m_{I_d}A^d_{m_{I_d}} \ar[ur]&
\end{tikzcd}
\end{center}

\noindent The ring $\mathcal{O}_{\mathcal{H}_n,I_d}$ then acts on ${(\mathscr{P}^n_{|I_d})}_p$ via $\kappa_{\mathcal{H}_n}(I_d) \simeq A^d_{m_{I_d}}/m_{I_d}A^d_{m_{I_d}}$. In particular ${(\mathscr{P}^n_{|I_d})}_p$ is a semisimple $\mathbb{C}[{(\mathbb{C}^2)}^n]^{\mathfrak{S}_n}_p$-module since the action of the ring $\mathbb{C}[{(\mathbb{C}^2)}^n]^{\mathfrak{S}_n}_p$ is defined using $\sigma^{\sharp}_{I_d}$ . Thanks to Lemma \ref{lemme_C2r_et}, one knows that the restriction of $h$ to ${{(\mathbb{C}^2)}^{r_d\ell}}^{\circ}$ is étale, which in particular implies that this morphism is unramified. Now Lemma \ref{lemme_etale_ss} with $R=\mathbb{C}[{(\mathbb{C}^2)}^n]_p$ and $S=\mathbb{C}[{{(\mathbb{C}^2)}^{r_d\ell}}^{\circ}]_q$ implies that ${(\mathscr{P}^n_{|I_d})}_p$ is a $\mathbb{C}[{{(\mathbb{C}^2)}^{r_d\ell}}^{\circ}]_q$-semisimple module. Finally, since ${(\mathscr{P}^n_{|I_d})}_p$ is a finite dimensional $\mathbb{C}[{(\mathbb{C}^2)}^n]$-module supported at $p$, the endomorphisms of ${(\mathscr{P}^n_{|I_d})}_p$ given by the action of $\left(X_i-X_i(p)\right)$ and of $\left(Y_i-Y_i(p)\right)$ are nilpotent for all $i \in \llbracket 1, n \rrbracket$; see e.g. \cite[II, $\S$4, no. 4, Corollary $1$]{BourbCA}.
 In particular, it follows that the endomorphisms of ${(\mathscr{P}^n_{|I_d})}_p$ given by the action of $\left(f^{\sharp}_n(X_i)-X_i(p)\right)$ and $\left(f^{\sharp}_n(Y_i)-Y_i(p)\right)$ are nilpotent for all $i \in \llbracket 1, n \rrbracket$. Combining semisimplicity with nilpotency gives the result.
The morphism $\textphnc{\Ahd}$ is by construction surjective.
\end{proof}

\noindent Recall that $\textphnc{\As}\colon S_p \xrightarrow{\,\smash{\raisebox{-0.65ex}{\ensuremath{\scriptstyle\sim}}}\,} \mathfrak{S}_{\mathrm{g}_{\Gamma}} \times \Gamma$ is an isomorphism of groups.
\begin{lemme}
\label{h_is_eq}
If $S_p$ acts on $\mathcal{O}_{\mathcal{X}_{\mathrm{g}_{\Gamma}},x^0}$ through $\textphnc{\As}$, then the morphism ${\textphnc{\Ahd} \colon \mathcal{O}_{\mathcal{X}_{\mathrm{g}_{\Gamma}},x^0} \to {(\mathscr{P}^n_{|I_d})}_p}$ is $S_p$-equivariant.
\end{lemme}
\begin{proof}
For each ${(\sigma,\gamma) \in \mathfrak{S}_{\mathrm{g}_{\Gamma}} \times \Gamma}$ and for each point $((I,u),(I',u')) \in \mathcal{X}_{\mathrm{g}_{\Gamma}} \times \mathcal{X}_{r_d\ell}$, define
\begin{equation*}
(\sigma,\gamma).((I,u),(I',u')):=\left((\gamma.I,\sigma\gamma u),(\gamma.I', \nabla(\gamma)u')\right).
\end{equation*}
This endows the variety $\mathcal{X}_{\mathrm{g}_{\Gamma}} \times \mathcal{X}_{r_d\ell}$ with an $(\mathfrak{S}_{\mathrm{g}_{\Gamma}} \times \Gamma)$-action. The morphism $\iota_{\mathrm{g}_{\Gamma}}$ is naturally $(\mathfrak{S}_{\mathrm{g}_{\Gamma}} \times \Gamma)$-equivariant since $J\in \mathcal{H}_{r_d\ell}^{\Gamma}$ and $q$ is $\nabla(\Gamma)$-invariant. By construction, the open set $V$ of ${(\mathbb{C}^2)}^n$ is $S_p$-stable and hence $S_p$ acts on $f_n^{-1}(V)$. Recall from \eqref{eq:betamap} that $\beta$ is the morphism mapping $((I,u),(I',u')) \in (f_{\mathrm{g}_{\Gamma}} \times f_{r_d\ell})^{-1}(V)$ to ${(I\cap I', (u,u')) \in f_n^{-1}(V)}$. For $(\sigma,\gamma) \in \mathfrak{S}_{\mathrm{g}_{\Gamma}} \times \Gamma$, we check that
\begin{align*}
\beta\left((\sigma,\gamma).\left((I,u),(I',u')\right)\right) & = \beta \left((\gamma.I,\sigma\gamma u),(\gamma.I', \nabla(\gamma)u')\right) \\
& = \left(\gamma.I \cap \gamma.I', (\sigma\gamma u, \nabla(\gamma)u')\right) \\
& = \left(\gamma. (I \cap I'), (\sigma\gamma u, \nabla(\gamma)u')\right) \\
& = \sigma \nabla(\gamma) . \beta \left((I,u),(I',u')\right)
\end{align*}
since
\[
\nabla(\gamma) . (u,u') = (\gamma u, \nabla(\gamma)u') \, \textrm{ for } \, (u,u' ) \in (\mathbb{C}^2)^{g_{\Gamma}} \times (\mathbb{C}^2)^{r_d \ell} = (\mathbb{C}^2)^n.
\]
Therefore, we deduce that $\alpha( g . x) = \textphnc{\As}(g) . \alpha(x)$ for $x \in f_n^{-1}(V)$ and $g \in S_p$. This implies that $\alpha^{\sharp}_{x^p}$ is $S_p$-equivariant. Finally, the fact that the affine open set $U^d$ has been taken to be $S_p$-stable and the fact that $I_d$ is $(\mathfrak{S}_n \times \Gamma)$-fixed, implies that $\textphnc{q}$ is $S_p$-equivariant.
We conclude that $\textphnc{\Ahd}\colon \mathcal{O}_{\mathcal{X}_{\mathrm{g}_{\Gamma}},x^0} \twoheadrightarrow ({\mathscr{P}^n_{|I_d}})_p$ is $S_p$-equivariant.
\end{proof}

\noindent Denote by $m_{I_{d_0}} \in \mathrm{Spec}(A^{d_0})$ the maximal ideal corresponding to $I_{d_0}$.
\begin{lemme}
\label{end_of_proof}
If ${({\mathscr{P}^n_{|I_d}})_p}^{\mathfrak{S}_{\mathrm{g}_{\Gamma}}}$ is a $1$-dimensional vector space, then the ideal $m_{I_{d_0}}\mathcal{O}_{\mathcal{X}_{\mathrm{g}_{\Gamma}},x^0}$ is contained in the annihilator $\mathrm{Ann}_{\mathcal{O}_{\mathcal{X}_{\mathrm{g}_{\Gamma}},x^0}}\left(({\mathscr{P}^n_{|I_d}})_p\right)$.
\end{lemme}
\begin{proof}
Recall that $A^{d_0}:=\mathcal{O}_{\mathcal{H}_{\mathrm{g}_{\Gamma}}}(U^{d_0})$.
 To show that ${m_{I_{d_0}}\mathcal{O}_{\mathcal{X}_{\mathrm{g}_{\Gamma}},x^0}\subset \mathrm{Ann}_{\mathcal{O}_{\mathcal{X}_{\mathrm{g}_{\Gamma}},x^0}}\left(({\mathscr{P}^n_{|I_d}})_p\right)}$, it is enough to show that the ideal ${\mathrm{Ann}_{A^{d_0}}\left(({\mathscr{P}^n_{|I_d}})_p\right)}$ is maximal since the $A^{d_0}$-module $({\mathscr{P}^n_{|I_d}})_p$ is supported at $I_{d_0}$.
Denote by $e \in {({\mathscr{P}^n_{|I_d}})_p}$ the identity element of this ring. Since $e$ is invariant under the action of $\mathfrak{S}_{g_{\Gamma}}$, our hypothesis forces ${({\mathscr{P}^n_{|I_d}})_p}^{\mathfrak{S}_{\mathrm{g}_{\Gamma}}} = \mathbb{C}.e$. Moreover, ${({\mathscr{P}^n_{|I_d}})_p}^{\mathfrak{S}_{\mathrm{g}_{\Gamma}}}$ is an $A^{d_0}$-submodule of $({\mathscr{P}^n_{|I_d}})_p$ since the group $\mathfrak{S}_{\mathrm{g}_{\Gamma}}$ acts trivially on $A^{d_0}$. One can check that $\mathrm{Ann}_{A^{d_0}}\left(({\mathscr{P}^n_{|I_d}})_p\right)=\mathrm{Ann}_{A^{d_0}}(\mathbb{C}.e)$. Finally, this implies that $\mathrm{Ann}_{A^{d_0}}\left(({\mathscr{P}^n_{|I_d}})_p\right)$ is a maximal ideal since the annihilator of a simple module is always maximal.
\end{proof}

\begin{proof}[Proof of Theorem \ref{reduction_thm}]
The algebraic variety $\mathcal{H}_n^{\Gamma,d}$ being an irreducible component of the scheme $\mathcal{H}_n^{\Gamma}$, on which $\mathfrak{S}_n \times \Gamma$ acts trivially, it is enough to prove this equality for $I=I_d$. The support of $\mathscr{P}^n_{|I_d}$ as an $\mathcal{O}_{\mathcal{X}_n}$-module is $\left\{(I_d,x) \in \mathcal{X}_n | \pi_n(x)=\sigma_n(I_d)\right\} = \rho_n^{-1}(I_d)$. Using \cite[II, $\S$4, no. 4, Proposition $19$]{BourbCA}, one has $\mathrm{Supp}_{\mathbb{C}[{(\mathbb{C}^2)}^n]}\left(\mathscr{P}^n_{|I_d}\right)=f_n\left(\rho^{-1}(I_d)\right)$. In particular, the support of $\mathscr{P}^n_{|I_d}$ as a $\mathbb{C}[{(\mathbb{C}^2)}^n]$-module is  an $\mathfrak{S}_n$-orbit which is $\Gamma$-stable, thus it is an ${(\mathfrak{S}_n \times \Gamma)}$-orbit.
Thanks to Lemma \ref{lemme_ind_fib}, one has
\[
\left[\mathscr{P}^n_{|I_d}\right]_{\mathfrak{S}_n\times \Gamma}=\left[\mathrm{Ind}_{S_p}^{\mathfrak{S}_n\times \Gamma}\left({(\mathscr{P}^n_{|I_d})}_p\right)\right]_{\mathfrak{S}_n \times \Gamma}.
\]
It remains to show that $\left[{(\mathscr{P}^n_{|I_d})}_p\right]_{S_p}=\left[\mathscr{P}^{g_{\Gamma}}_{|I_{d_0}}\right]_{S_p}$. We first note that repeating the above argument with $\mathfrak{S}_n$ rather than $\mathfrak{S}_n \times \Gamma$ shows that
\begin{equation}\label{eq:indSnprocei}
\left[\mathscr{P}^n_{|I_d}\right]_{\mathfrak{S}_n}=\left[\mathrm{Ind}_{\mathfrak{S}_{g_{\Gamma}}}^{\mathfrak{S}_n} \left({(\mathscr{P}^n_{|I_d})}_p\right)\right]_{\mathfrak{S}_n},
\end{equation}
since the stablizer of $p$ in $\mathfrak{S}_n$ is $\mathfrak{S}_{g_{\Gamma}}$. This  implies that $\mathrm{Ind}_{\mathfrak{S}_{\mathrm{g}_{\Gamma}}}^{\mathfrak{S}_n}\left({(\mathscr{P}^n_{|I_d})}_p\right)$ is isomorphic to the regular representation of $\mathfrak{S}_n$.
Combining Proposition \ref{h_exists} and Lemma \ref{h_is_eq}, one has an $S_p$-equivariant surjective morphism $\textphnc{\Ahd}\colon \mathcal{O}_{\mathcal{X}_{\mathrm{g}_{\Gamma}},x^0} \twoheadrightarrow {(\mathscr{P}^n_{|I_d})}_p$  such that the diagram
\begin{center}
\begin{tikzcd}
\mathcal{O}_{\mathcal{X}_{\mathrm{g}_{\Gamma}}\times \mathcal{X}_{r_d\ell},x^{(0,q)}} \ar[r,"\alpha^{\sharp}_{x^p}"',"\sim"] \ar[d,"\iota^{\sharp}_{x^0}"', two heads]  & \mathcal{O}_{\mathcal{X}_n,x^p}\ar[r, two heads, "\textphnc{q}"] & {(\mathscr{P}^n_{|I_d})}_p \\
\mathcal{O}_{\mathcal{X}_{\mathrm{g}_{\Gamma}},x^0} \ar[urr,"\textphnc{\Ahd}"'] & &
\end{tikzcd}
\end{center}

\noindent is commutative. Since $\mathrm{Ind}_{\mathfrak{S}_{\mathrm{g}_{\Gamma}}}^{\mathfrak{S}_n}\left({(\mathscr{P}^n_{|I_d})}_p\right)$ is isomorphic to the regular representation of $\mathfrak{S}_n$, the space ${\left(({\mathscr{P}^n_{|I_d}})_p\right)}^{\mathfrak{S}_{\mathrm{g}_{\Gamma}}}$ must be one-dimensionnal. Therefore, Lemma \ref{end_of_proof} says that $m_{I_{d_0}}\mathcal{O}_{\mathcal{X}_{\mathrm{g}_{\Gamma}},x^0}\subset \mathrm{Ann}_{\mathcal{O}_{\mathcal{X}_{\mathrm{g}_{\Gamma}},x^0}}\left(({\mathscr{P}^n_{|I_d}})_p\right)$. Hence we can factor the morphism $\textphnc{\Ahd}$ as
\[
\begin{tikzcd}
\mathcal{O}_{\mathcal{X}_{\mathrm{g}_{\Gamma}},x^0} \ar[r,"\textphnc{\Ahd}", two heads] \ar[d, two heads]& ({\mathscr{P}^n_{|I_d}})_p\\
\mathcal{O}_{\mathcal{X}_{\mathrm{g}_{\Gamma}},x^0}/m_{I_{d_0}}\mathcal{O}_{\mathcal{X}_{\mathrm{g}_{\Gamma}},x^0} \ar[ur, "\hat{\textphnc{\Ahd}}"',two heads]&
\end{tikzcd}
\]
As shown in Lemma~\ref{lemme_supp_0}, $\sigma_{g_{\Gamma}}(I_{d_0}) = \overline{0}$ and hence the fiber $(\mathscr{P}^{\mathrm{g}_{\Gamma}}_{|I_{d_0}})$ is supported at ${0 \in (\mathbb{C}^2)^{g_{\Gamma}}}$ when considered as a $\mathbb{C}[{(\mathbb{C}^2)}^{\mathrm{g}_{\Gamma}}]$-module. Since $x^0 = (I_{d_0},0)$, this implies that the localization map
\[
\mathscr{P}^{\mathrm{g}_{\Gamma}}_{|I_{d_0}} \to {(\mathscr{P}^{\mathrm{g}_{\Gamma}}_{|I_{d_0}})}_0 = \mathcal{O}_{\mathcal{X}_{\mathrm{g}_{\Gamma}},x^0}/m_{I_{d_0}}\mathcal{O}_{\mathcal{X}_{\mathrm{g}_{\Gamma}},x^0}
\]
is an isomorphism. This identification is $(\mathfrak{S}_{g_{\Gamma}} \times \Gamma)$-equivariant. Making $S_p$ acts via the isomorphism $\textphnc{\As}$, we may think of it as a $S_p$-equivariant isomorphism. In particular, the quotient $\mathcal{O}_{\mathcal{X}_{\mathrm{g}_{\Gamma}},x^0}/m_{I_{d_0}}\mathcal{O}_{\mathcal{X}_{\mathrm{g}_{\Gamma}},x^0}$ has dimension $\mathrm{g}_{\Gamma}!$. Equation \eqref{eq:indSnprocei} implies that $({\mathscr{P}^n_{|I_d}})_p$ also has dimension $\mathrm{g}_{\Gamma}!$. We deduce that the surjection $\hat{\textphnc{\Ahd}}$ is actually an isomorphism. Finally, we conclude that $\mathscr{P}^{\mathrm{g}_{\Gamma}}_{|I_{d_0}}$ is isomorphic to $({\mathscr{P}^n_{|I_d}})_p$ as $S_p$-modules, provided $S_p$ acts on the former via $\textphnc{\As}$.
\end{proof}

\begin{rmq}
The diagonal copy of $\mathbb{G}_{\mathrm{m}}$ in $\mathrm{GL}_2(\mathbb{C})$ commutes with the action of $\Gamma$. Therefore $\mathscr{P}^n$ can be considered as a $(\mathfrak{S}_n\times \Gamma \times \mathbb{G}_{\mathrm{m}})$-equivariant vector bundle on $\mathcal{H}_n$ and  on $\mathcal{H}_n^{\Gamma,d}$. However, our methods do not allow for a reduction result that induces the $\mathbb{G}_{\mathrm{m}}$-action. Indeed, the action of $\mathbb{G}_{\mathrm{m}}$ on $\mathcal{H}_n^{\Gamma,d}$ is non-trivial so one cannot expect a reduction result to hold.
\end{rmq}

\section{Combinatorial consequences in type $A$}
In this section, we explore the meaning of Theorem \ref{reduction_thm} when $\Gamma$ is of type $A$.
Fix an integer $\ell \geq 1$. Recall that $\zeta_{\ell}$ denotes the primitive $\ell^{\text{th}}$ root of unity $e^{\frac{2i\pi}{\ell}}$, and that ${\omega_{\ell}\in \mathrm{SL}_2(\mathbb{C})}$ is the diagonal matrix $\mathrm{diag}( \zeta_{\ell}, \zeta_{\ell}^{-1})$. The cyclic subgroup of order $\ell$ in $\mathrm{SL}_2(\mathbb{C})$ is  ${\mu_{\ell}= \langle \omega_{\ell} \rangle}$.
Assume, in this section, that ${\Gamma=\mu_{\ell}}$.

Let us first fix notation concerning partitions. A partition $\lambda$ of $n$, denoted by $\lambda \vdash n$, is a tuple $(\lambda_1 \geq \lambda_2 \geq ... \geq \lambda_r \geq 0)$ such that $|\lambda|:=\sum_{i=1}^r{\lambda_i}=n$. Denote by $\mathcal{P}_n$ the set of all partitions of $n$ and by $\mathcal{P}$ the set of all partitions of integers. For ${\lambda=(\lambda_1,...,\lambda_r) \in \mathcal{P}}$, denote by ${\mathcal{Y}(\lambda)}$ its associated Young diagram ${\left\{(i,j) \in \mathbb{Z}_{\geq 0}^2| i < \lambda_j, j < r\right\}}$. The conjugate partition of a partition $\lambda$ of $n$, denoted by $\lambda^*$, is the partition associated with the reflection of $\mathcal{Y}(\lambda)$ along the diagonal (which is again a Young diagram of a partition of $n$). For example, consider $\lambda=(2,2,1)$. Its associated Young diagram is
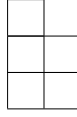
\begin{center}
\yng(1,2,2)
\end{center}
and in that case  $\lambda^*=(3,2)$.
A partition $\lambda$ will be called symmetric if it is equal to its conjugate.
The hook $H_{(i,j)}(\lambda)$ in position $(i,j) \in \mathcal{Y}(\lambda)$ of a partition $\lambda$ is the set
\begin{equation*}
\left\{(a,b) \in \mathcal{Y}(\lambda) | a=i \text{ and } b \geq j \text{ or } a > i \text{ and } b=j\right\}.
\end{equation*}
Define the length $h_{(i,j)}(\lambda)$ of a hook $H_{(i,j)}(\lambda)$ to be its cardinal. In addition, let $n(\lambda)$ denote the partition statistic $\sum_{i=1}^{|\lambda|}{(i\smin1)\lambda_i}$ of $\lambda$.

Let $\gamma_{\ell}(\lambda)$ be the $\ell$-core of the partition $\lambda \in \mathcal{P}$, which is the partition obtained from $\lambda$ by removing all hooks of length $\ell$. Denote by $\mathrm{g}_{\ell}(\lambda):={|\gamma_{\ell}(\lambda)|}$ and by ${r_{\ell}(\lambda):=\frac{|\lambda|-\mathrm{g}_{\ell}(\lambda)}{\ell}}$ the number of hooks of length $\ell$ that one needs to remove from $\lambda$ to obtain $\gamma_{\ell}(\lambda)$. If $\lambda$ is clear from context, we shorten $\gamma_{\ell}(\lambda)$, $\mathrm{g}_{\ell}(\lambda)$ and $r_{\ell}(\lambda)$ to $\gamma_{\ell}$, $\mathrm{g}_{\ell}$ and $r_{\ell}$. Given $\mathrm{g}_{\ell}$ and $r_{\ell}$ such that $n = \mathrm{g}_{\ell} +  r_{\ell} \ell$, we associate the permutation
\begin{equation}\label{eq:wperm}
w_{\ell,n}^{\mathrm{g}_{\ell}} := (\mathrm{g}_{\ell}+1,\dots ,\mathrm{g}_{\ell}+\ell) \dots (n-\ell+1,\dots ,n) \in \mathfrak{S}_n,
\end{equation}
which is a product of $r_{\ell}$ cycles of length $\ell$. Let $C_{\ell,n}$ be the cyclic subgroup of $\mathfrak{S}_{r_{\ell}\ell}$ generated by $w_{\ell,n}^{\mathrm{g}_{\ell}}$. Consider also the subgroup $W_{\ell,n}^{\mathrm{g}_{\ell}}:= \mathfrak{S}_{\mathrm{g}_{\ell}} \times C_{\ell,n}$ of $\mathfrak{S}_n$. Denote by $\theta_{\ell}$ the character of $C_{\ell,n}$ such that $\theta_{\ell}(w_{\ell,n}^{\mathrm{g}_{\ell}})=\zeta_{\ell}$.   Let us also use the following notation. For $V$ a given $W_{\ell,n}^{\mathrm{g}_{\ell}}$-module let
\begin{equation*}
[V]_{W_{\ell,n}^{\mathrm{g}_{\ell}}}=\sum_{j=0}^{\ell \smin 1}{[V^{\ell}_j\boxtimes \theta_{\ell}^j]_{W_{\ell,n}^{\mathrm{g}_{\ell}}}},
\end{equation*}
where $V^{\ell}_j$ is an $\mathfrak{S}_{\mathrm{g}_{\ell}}$-module for each $j \in \llbracket 0, \ell \smin 1 \rrbracket$ and $\boxtimes$ denotes the external tensor product. Moreover, if $\lambda$ is a partition of $n$,  let us shorten $\mathscr{P}^n_{|I_{\lambda}}$, the fiber of the $n^{\text{th}}$-Procesi bundle at the monomial ideal $I_{\lambda}$ generated by $\left\{x^iy^j | (i,j)\in \mathbb{N}^2\setminus{\mathcal{Y}(\lambda)}\right\}$, to $\mathscr{P}^n_{\lambda}$.

\subsection{Corollary of the reduction theorem}

To state the main result of this subsection, we need two lemmas. Denote by $\tau_{\ell}$ the character of $\mu_{\ell}$ such that $\tau_{\ell}(\omega_{\ell})=\zeta_{\ell}$.

\begin{lemme}
\label{lemme_dcyc}
Let $C_1$ and $C_2$ be two groups isomorphic to $\mu_{\ell}$. Take $c_1\in C_1$ and $c_2\in C_2$ generators of  $C_1$ and $C_2$. If one denotes respectively by $\tau_1$, $\tau_2$ and $\tau_3$ the characters of respectively $C_1$, $C_2$ and $\langle (c_1,c_2) \rangle < C_1\times C_2$ that respectively map $c_1$, $c_2$ and $(c_1,c_2)$ to $\zeta_{\ell}$, then
\begin{equation*}
 \mathrm{Ind}_{\langle (c_1,c_2) \rangle}^{C_1\times C_2}\left(\tau_3^j\right)=\sum_{i=0}^{\ell \smin 1}{\tau_1^{j-i}\boxtimes \tau_2^i}, \quad \forall j \in \llbracket 0 , \ell \smin 1\rrbracket.
\end{equation*}
\end{lemme}
\begin{proof}
Take $(p,q) \in \llbracket 0,\ell \smin 1\rrbracket^2$. On the one hand, Frobenius reciprocity gives
\begin{align*}
\langle \tau_1^p\boxtimes \tau_2^q, \mathrm{Ind}_{\langle (c_1,c_2) \rangle}^{C_1 \times C_2}\left(\tau_3^j\right)\rangle &= \langle \mathrm{Res}_{\langle (c_1,c_2) \rangle}^{C_1\times C_2}\left(\tau_1^p\boxtimes \tau_2^q\right), \tau_3^j\rangle\\
                &=\langle \tau_3^{p+q},\tau_3^j\rangle\\
                &=\delta^{j}_{p+q} \text{ }.
\end{align*}
On the other hand
\begin{align*}
\sum_{i=0}^{\ell \smin 1}{\langle \tau_1^p \boxtimes \tau_2^q,\tau_1^{j-i}\boxtimes \tau_2^{i}\rangle} &= \sum_{i=0}^{\ell \smin 1}{\delta^{p}_{j\smin i}\delta^i_q}\\
&= \delta_{j-q}^p \text{ }.
\end{align*}
\end{proof}

\begin{lemme}
\label{lemme_dcyc2}
Let $\Delta$ be the cyclic subgroup of $\mathfrak{S}_{r_{\ell}\ell}\times \mu_{\ell}$ generated by the element $(w_{\ell,n}^{\mathrm{g}_{\ell}},\omega_{\ell})$. If $\hat{\theta}_{\ell}$ denotes the character of $\Delta$ such that $\hat{\theta}_{\ell}\big((w_{\ell,n}^{\mathrm{g}_{\ell}},\omega_{\ell})\big)=\zeta_{\ell}$ then
\begin{equation*}
 \mathrm{Ind}_{\Delta}^{\mathfrak{S}_{r_{\ell}\ell}\times \mu_{\ell}}\left(\hat{\theta}_{\ell}^j\right)= \sum_{i=0}^{\ell \smin 1}{\mathrm{Ind}_{C_{\ell,n}}^{\mathfrak{S}_{r_{\ell}\ell}}\left(\theta_{\ell}^{j-i}\right)\boxtimes \tau_{\ell}^i}, \quad \forall j \in \llbracket 0, \ell \smin 1\rrbracket.
\end{equation*}
\end{lemme}
\begin{proof}
One has $\mathrm{Ind}_{\Delta}^{\mathfrak{S}_{r_{\ell}\ell}\times \mu_{\ell}}\left(\hat{\theta}_{\ell}^j\right)=\mathrm{Ind}_{C_{\ell,n} \times \mu_{\ell}}^{\mathfrak{S}_{r_{\ell}\ell}\times \mu_{\ell}}\left(\mathrm{Ind}_{\Delta}^{C_{\ell,n} \times \mu_{\ell}}\left(\hat{\theta}_{\ell}^j\right)\right)$. Using Lemma \ref{lemme_dcyc},
\begin{align*}
\mathrm{Ind}_{C_{\ell,n} \times \mu_{\ell}}^{\mathfrak{S}_{r_{\ell}\ell}\times \mu_{\ell}}\left(\mathrm{Ind}_{\Delta}^{C_{\ell,n} \times \mu_{\ell}}\left(\hat{\theta}_{\ell}^j\right)\right) &= \sum_{i=0}^{\ell \smin 1}{\mathrm{Ind}_{C_{\ell,n} \times \mu_{\ell}}^{\mathfrak{S}_{r_{\ell}\ell}\times \mu_{\ell}}\left(\theta_{\ell}^{j-i}\boxtimes \tau_{\ell}^i\right)}\\
&=\sum_{i=0}^{\ell \smin 1}{\mathrm{Ind}_{C_{\ell,n}}^{\mathfrak{S}_{r_{\ell}\ell}}\left(\theta_{\ell}^{j-i}\right)\boxtimes \tau_{\ell}^i}.
\end{align*}
\end{proof}

\noindent We can now state and prove the main result of this subsection.
\begin{cor}
\label{cor1_a}
For each partition $\lambda$ of $n$, one has the following decomposition of $\mathscr{P}^n_{\lambda}$:
\begin{equation*}
[\mathscr{P}^n_{\lambda}]_{\mathfrak{S}_n\times \mu_{\ell}}=\sum_{i=0}^{\ell \smin 1}{\sum_{j=0}^{\ell \smin 1}{\left[\mathrm{Ind}_{W_{\ell,n}^{\mathrm{g}_{\ell}}}^{\mathfrak{S}_n}\left((\mathscr{P}^{\mathrm{g}_{\ell}}_{\gamma_{\ell}})^{\ell}_j\boxtimes\theta_{\ell}^{i-j}\right)\boxtimes \tau_{\ell}^i\right]_{\mathfrak{S}_n\times \mu_{\ell}}}},
\end{equation*}
\end{cor}
\begin{proof}
With the notation established at the beginning of this section, the group $S_p$ introduced in section \ref{sec_reduction} is equal to $\mathfrak{S}_{\mathrm{g}_{\ell}}\times \Delta$. Thanks to Theorem \ref{reduction_thm}, it is enough to show that
\begin{equation*}
\left[\mathrm{Ind}_{S_p}^{\mathfrak{S}_n\times \mu_{\ell}}(\mathscr{P}^{\mathrm{g}_{\ell}}_{\gamma_{\ell}})\right]_{\mathfrak{S}_n \times \mu_{\ell}} = \sum_{i=0}^{\ell \smin 1}{\sum_{j=0}^{\ell \smin 1}{\left[\mathrm{Ind}_{W_{\ell,n}^{\mathrm{g}_{\ell}}}^{\mathfrak{S}_n}\left((\mathscr{P}^{\mathrm{g}_{\ell}}_{\gamma_{\ell}})^{\ell}_j\boxtimes\theta_{\ell}^{i-j}\right)\boxtimes \tau_{\ell}^i\right]_{\mathfrak{S}_n\times \mu_{\ell}}}}.
\end{equation*}
One has
\begin{equation*}
\left[\mathrm{Ind}_{S_p}^{\mathfrak{S}_n\times \mu_{\ell}}(\mathscr{P}^{\mathrm{g}_{\ell}}_{\gamma_{\ell}})\right]_{\mathfrak{S}_n\times \mu_{\ell}}=\sum_{j=0}^{\ell \smin 1}{\left[\mathrm{Ind}_{\mathfrak{S}_{\mathrm{g}_{\ell}}\times \Delta}^{\mathfrak{S}_n\times \mu_{\ell}}\left((\mathscr{P}^{\mathrm{g}_{\ell}}_{\gamma_{\ell}})^{\ell}_j\boxtimes \hat{\theta}_{\ell}^j\right)\right]_{\mathfrak{S}_n \times \mu_{\ell}}}.
\end{equation*}
Moreover
\begin{align*}
\sum_{j=0}^{\ell \smin 1}{\mathrm{Ind}_{\mathfrak{S}_{\mathrm{g}_{\ell}}\times \Delta}^{\mathfrak{S}_n\times \mu_{\ell}}\left((\mathscr{P}^{\mathrm{g}_{\ell}}_{\gamma_{\ell}})^{\ell}_j\boxtimes \hat{\theta}_{\ell}^j\right)}
&=\sum_{j=0}^{\ell \smin 1}{\mathrm{Ind}_{\mathfrak{S}_{\mathrm{g}_{\ell}}\times \mathfrak{S}_{r_{\ell}l} \times \mu_{\ell}}^{\mathfrak{S}_n\times \mu_{\ell}}\left(\mathrm{Ind}_{\mathfrak{S}_{\mathrm{g}_{\ell}}\times \Delta}^{\mathfrak{S}_{\mathrm{g}_{\ell}}\times \mathfrak{S}_{r_{\ell}\ell} \times \mu_{\ell}}\left( (\mathscr{P}^{\mathrm{g}_{\ell}}_{\gamma_{\ell}})^{\ell}_j\boxtimes \hat{\theta}_{\ell}^j\right)\right)}\\
&=\sum_{j=0}^{\ell \smin 1}{\mathrm{Ind}_{\mathfrak{S}_{\mathrm{g}_{\ell}}\times \mathfrak{S}_{r_{\ell}\ell} \times \mu_{\ell}}^{\mathfrak{S}_n\times \mu_{\ell}}\left((\mathscr{P}^{\mathrm{g}_{\ell}}_{\gamma_{\ell}})^{\ell}_j \boxtimes \mathrm{Ind}_{\Delta}^{\mathfrak{S}_{r_{\ell}l} \times \mu_{\ell}}\left(\hat{\theta}_{\ell}^j\right)\right)}\\
&=\sum_{i=0}^{\ell \smin 1}{\sum_{j=0}^{\ell \smin 1}{\mathrm{Ind}_{\mathfrak{S}_{\mathrm{g}_{\ell}}\times \mathfrak{S}_{r_{\ell}\ell} \times \mu_{\ell}}^{\mathfrak{S}_n\times \mu_{\ell}}\left((\mathscr{P}^{\mathrm{g}_{\ell}}_{\gamma_{\ell}})^{\ell}_j \boxtimes \mathrm{Ind}_{C_{\ell,n}}^{\mathfrak{S}_{r_{\ell}\ell}}\left(\theta_{\ell}^{j-i}\right)\boxtimes \tau_{\ell}^i\right)}}.
\end{align*}
The last equality follows from Lemma \ref{lemme_dcyc2}. By gathering terms, one has
\begin{equation*}
\sum_{i=0}^{\ell \smin 1}{\sum_{j=0}^{\ell \smin 1}{\mathrm{Ind}_{\mathfrak{S}_{\mathrm{g}_{\ell}}\times \mathfrak{S}_{r_{\ell}l} \times \mu_{\ell}}^{\mathfrak{S}_n\times \mu_{\ell}}\left((\mathscr{P}^{\mathrm{g}_{\ell}}_{\gamma_{\ell}})^{\ell}_j \boxtimes \mathrm{Ind}_{C_{\ell,n}}^{\mathfrak{S}_{r_{\ell}l}}\left(\theta_{\ell}^{j-i}\right)\boxtimes \tau_{\ell}^i\right)}}=
\sum_{i=0}^{\ell \smin 1}{\sum_{j=0}^{\ell \smin 1}{\mathrm{Ind}_{W_{\ell,n}^{\mathrm{g}_{\ell}}}^{\mathfrak{S}_n}\left((\mathscr{P}^{\mathrm{g}_{\ell}}_{\gamma_{\ell}})^{\ell}_j \boxtimes \theta_{\ell}^{j-i}\right)\boxtimes \tau_{\ell}^i}}
\end{equation*}
Finally, since every representation of $\mathfrak{S}_n$ is isomorphic to its dual, for each ${(i,j) \in \llbracket 0, \ell \smin 1 \rrbracket^2}$, one has
\begin{equation*}
\left[\mathrm{Ind}_{W_{\ell,n}^{\mathrm{g}_{\ell}}}^{\mathfrak{S}_n}\left((\mathscr{P}^{\mathrm{g}_{\ell}}_{\gamma_{\ell}})^{\ell}_j \boxtimes \theta_{\ell}^{j-i}\right)\right]_{\mathfrak{S}_n}=\left[\mathrm{Ind}_{W_{\ell,n}^{\mathrm{g}_{\ell}}}^{\mathfrak{S}_n}\left((\mathscr{P}^{\mathrm{g}_{\ell}}_{\gamma_{\ell}})^{\ell}_j \boxtimes \theta_{\ell}^{i-j}\right)\right]_{\mathfrak{S}_n}
\end{equation*}
\end{proof}

\begin{rmq}
If one takes $\lambda=\gamma_{\ell}$, then $r_{\ell}=0$ and $W_{\ell,n}^{\mathrm{g}_{\ell}}=\mathfrak{S}_n$. In that case, Corollary \ref{cor1_a} is trivially true and does not provide any additional information. Note also that Corollary \ref{cor1_a} implies \cite[Theorem $4.6$]{BLM06} when the complex reflection group is taken to be $\mathfrak{S}_n$ and $\Gamma$ is taken to be $C_{\ell,n}$.
\end{rmq}

\subsection{Independant proofs in two edge cases}
\noindent The proof of Theorem~\ref{reduction_thm} relies heavily on the geometry of the punctual Hilbert scheme and the deep result of Haiman on the isospectral Hilbert scheme. The goal of what follows is to prove Corollary~\ref{cor1_a} directly in two special cases without using Theorem \ref{reduction_thm}. To prove Corollary \ref{cor1_a} in these two edge cases, we  use the representation theory of the symmetric group and symmetric functions. We will in particular use \cite[Theorem $4.6$]{BLM06}. The irreducible representations of $\mathfrak{S}_n$ are parametrized by partitions of $n$. Denote respectively by $V_{\lambda}$ and $\chi_{\lambda}$ the representation space  and the character of the irreducible representation of $\mathfrak{S}_n$ associated with $\lambda \vdash n$.
\begin{deff}
Let $R$ be any finitely generated $\mathbb{Z}$-algebra. For a given integer $k$, define the ring of symmetric polynomials over $R$ as $\Lambda^k_R:=R[z_1,...,z_k]^{\mathfrak{S}_k}$. Setting $\mathrm{deg}(z_i)=1$ for all $i \in \llbracket 1,k \rrbracket$,  $\Lambda_R^k=\bigoplus_{d \geq 0}{\Lambda^k_{R,d}}$ is a graded ring. One has moreover a ring morphism $\pi^k\colon \Lambda^{k+1}_R \to \Lambda^k_R$ by mapping $z_{k+1}$ to $0$. For each integer $d$, the morphism $\pi^k$ restricts to a morphism $\pi^k_d\colon \Lambda^{k+1}_{R,d} \to \Lambda^k_{R,d}$ of $R$-modules. One can now define the graded $R$-algebra of symmetric functions.
\begin{equation*}
\Lambda_R := \bigoplus_{d \geq 0}{\varprojlim\Lambda^k_{R,d}}
\end{equation*}
\end{deff}

\noindent In the following, we will shorten $\Lambda_{\mathbb{Z}}$ to $\Lambda$. Let us recall the notation concerning symmetric functions. For $\mu \in \mathcal{P}$ a given partition, denote by $p_{\mu}$ and $s_{\mu}$ respectively the power symmetric function and the Schur function associated with $\mu$. We recall now the plethystic substitution. One knows that $\Lambda\otimes_{\mathbb{Z}} \mathbb{Q}$ is generated as a free $\mathbb{Q}$-algebra by the family $\{p_k | k \in \mathbb{Z}_{\geq 0}\}$.
\begin{deff}
Take $K$ a finitely generated field extension of $\mathbb{Q}$. Take $\{s_1,\dots,s_m\}$ a set of generators of $K$ i.e. $K=\mathbb{Q}(s_1,\dots, s_m)$. For $A \in \Lambda_K:=\Lambda\otimes_{\mathbb{Z}} K$, and $k \in \mathbb{Z}_{\geq 0}$ define $p_k\big[A\big]$ to be the symmetric function in the indeterminates $s_1^k,\dots s_m^k,z_1^k,z_2^k,\dots$ . One can now extend the plethystic substitution to the following endomorphism
\begin{center}
$\big[A\big]\colon \begin{array}{ccc}
\Lambda_K & \to & \Lambda_K\\
f & \mapsto & f\big[A\big]  \\
\end{array}$.
\end{center}
\end{deff}

\begin{rmq}
Mainly, we will do plethystic substitutions using $Z:= p_1 =\sum_{k \geq 1}{z_k} \in \Lambda$. Note that for all $k\geq1, p_k[Z]=p_k$ and so for all $f \in \Lambda_K, f[Z]=f$.
\end{rmq}

\noindent For ${\left([V],[W]\right) \in \mathcal{R}(\mathfrak{S}_{k_1})\times \mathcal{R}(\mathfrak{S}_{k_2})}$ define the induced product
\begin{equation*}
{[V].[W]:=\left[\mathrm{Ind}_{\mathfrak{S}_{k_1} \times \mathfrak{S}_{k_2}}^{\mathfrak{S}_{k_1+k_2}}(V\otimes W)\right]}.
\end{equation*}
This product endows $\mathcal{R}(\mathfrak{S}):=\bigoplus_{k \geq 0}{\mathcal{R}(\mathfrak{S}_k)}$ and $\mathcal{R}^{\mathrm{gr}}(\mathfrak{S}):=\bigoplus_{k \geq 0}{\mathcal{R}^{\mathrm{gr}}(\mathfrak{S}_k)}$ with the structure of graded rings. Let us denote by  $\mathrm{Fr}\colon \mathcal{R}(\mathfrak{S}) \xrightarrow{\,\smash{\raisebox{-0.65ex}{\ensuremath{\scriptstyle\sim}}}\,} \Lambda$ the Frobenius characteristic map which is an isomorphism of graded rings. If $A:=\bigoplus_{(r,s) \in \mathbb{Z}^2}{A_{r,s}}$ is a bigraded $\mathfrak{S}_n$-module, denote by $\mathrm{Fr}(A)$ the following element ${\sum_{(r,s) \in \mathbb{Z}^2}{\mathrm{Fr}(A_{r,s})q^rt^s}}$ of $\Lambda[q^{\pm 1},t^{\pm 1}]$.

\begin{rmq}
Graded $\mathfrak{S}_n$-modules will be considered bigraded with trivial $t$-graduation.
\end{rmq}

\begin{deff}
Take $(F,G) \in \Lambda^2$ and write $[V]=\mathrm{Fr}^{-1}(F), [W]=\mathrm{Fr}^{-1}(G)$. The Kronecker product of $F$ and $G$ is
\begin{equation*}
F\otimes G:=\mathrm{Fr}([V]\otimes [W])
\end{equation*}
\end{deff}

\noindent If $\lambda$ is a partition of $n$, the fiber $\mathscr{P}^n_{\lambda}$ is a bigraded $\mathfrak{S}_n$-module. Haiman introduced the transformed Macdonald symmetric functions $\tilde{H}_{\lambda}(z;q,t)$ \cite[Definition $3.5.2$]{H03}. The $n!$ theorem (\cite[Theorem $4.1.5$]{H03}) can be reformulated the following way.
\begin{prop}
\label{lien_proc_mcdo}
For each partition $\lambda$ of $n$, one has $\mathrm{Fr}([\mathscr{P}^n_{\lambda}])=\tilde{H}_{\lambda}(z;q,t)$.
\end{prop}

\begin{deff}
Let $V$ be a finite-dimensional complex vector space and $G$ be a finite subgroup of $\mathrm{GL}(V)$ generated by (pseudo-)reflections in $V$. The group $G$ then acts on the symmetric algebra $S(V)$ of $V$, which is naturally graded $S(V)=\bigoplus_{i\geq 0}{S^i(V)}$. Let $\mathfrak{M}$ be the graded maximal ideal of $S(V)^G$. Define $S(V)^{\mathrm{co}(G)}:=S(V)/\mathfrak{M}S(V)$ the coinvariant algebra of $G$, which is then also graded. Note that as a $G$-module it is isomorphic to the regular representation of $G$ by the Chevalley-Shephard-Todd Theorem.
\end{deff}

\noindent If $V=\bigoplus_{i \in \mathbb{Z}}{V_i}$ is a graded vector space, then let $\mathrm{dim}^{\mathrm{gr}}(V):=\sum_{i \in \mathbb{Z}}{\mathrm{dim}(V_i)q^i} \in \mathbb{Z}[q^{\pm 1}]$ be the graded dimension of $V$. In this section, let us denote by $V^n=\mathbb{C}^n$ the permutation representation of $\mathfrak{S}_n$.

\begin{deff}
For  $\lambda \vdash n$, define
\begin{equation*}
{F_{\lambda}(q):=\mathrm{dim}^{\mathrm{gr}}\big((S(V^n)^{\mathrm{co}(\mathfrak{S}_n)}\otimes V_{\lambda}^*)^{\mathfrak{S}_n}\big)}
\end{equation*}
the fake degree associated with the irreducible representation $V_{\lambda}$ of $\mathfrak{S}_n$.
\end{deff}

\begin{lemme}
\label{fk_deg}
If $\lambda \vdash n$, then the fake degree $F_{\lambda}(q)$ is equal to $q^{n(\lambda)}\frac{\prod_{i=1}^n{\left(1-q^i\right)}}{\prod_{c \in \mathcal{Y}(\lambda)}{\left(1-q^{h_c(\lambda)}\right)}}$.
\end{lemme}
\begin{proof}
To prove this equality one can use \cite[Proposition $4.11$]{Stan79} and \cite[Corollary $7.21.5$]{Stan99}.
\end{proof}

\noindent Let us first study $[\mathscr{P}^n_{\lambda}]$ as a $(\mathfrak{S}_n \times \mathbb{T}_1)$-module, where $\mathbb{T}_1$ denotes the maximal diagonal torus of $\mathrm{SL}_2(\mathbb{C})$. Using \cite[Proposition $3.5.10$]{H03}, one has
\begin{equation}
\tilde{H}_{\lambda}(z;q,q^{-1})= \frac{\prod_{c \in \mathcal{Y}(\lambda)}{\left(1-q^{h_c(\lambda)}\right)}}{q^{n(\lambda)}}s_{\lambda}\left[\frac{Z}{1-q}\right] \tag{$\star$}.
\end{equation}

\begin{lemme}
\label{plethysm_s}
The following equality holds in $\Lambda_{\mathbb{Q}(q)}$:
\begin{equation*}
s_{\lambda}\left[\frac{Z}{1-q}\right]=\frac{\mathrm{Fr}\left(S(V^n)^{\mathrm{co}(\mathfrak{S}_n)}\otimes V_{\lambda}\right)}{\prod_{i=1}^n{\left(1-q^i\right)}}.
\end{equation*}
\end{lemme}
\begin{proof}
Let us start rewriting the plethysm
\begin{align*}
s_{\lambda}\left[\frac{Z}{1-q}\right] &= \sum_{i=0}^{\infty}{\mathrm{Fr}\left([S^i(V^n)\otimes V_{\lambda}]\right)q^i}\\
                       &= \sum_{i=0}^{\infty}{\mathrm{Fr}\left([S^i(V^n)]\right)\otimes \mathrm{Fr}\left([V_{\lambda}]\right)q^i}\\
                       &= \sum_{i=0}^{\infty}{\mathrm{Fr}\left([S^i(V^n)]\right)q^i} \otimes \mathrm{Fr}([V_{\lambda}])\\
                       &= s_n\left[\frac{Z}{1-q}\right]\otimes \mathrm{Fr}\left([V_{\lambda}]\right)
                     \end{align*}
where the first and the last equalities come from \cite[Proposition $3.3.1$]{H03}. Proposition \ref{lien_proc_mcdo} for $\lambda=(n)$ gives
\begin{equation*}
\tilde{H}_{n}(z;q,t)=\tilde{H}_{n}(z;q,q^{-1})=\mathrm{Fr}\left(\left[S(V^n)^{\mathrm{co}(\mathfrak{S}_n)}\right]\right) .
\end{equation*}
Moreover using the equation $(\star)$, one has
\begin{equation*}
\tilde{H}_{n}(z;q,q^{-1})=\prod_{i=1}^n{\left(1-q^i\right)}s_n\left[\frac{Z}{1-q}\right] .
\end{equation*}
Summing it up, one gets
\begin{align*}
s_{\lambda}\left[\frac{Z}{1-q}\right] &= \frac{\mathrm{Fr}\left(\left[S(V^n)^{\mathrm{co}(\mathfrak{S}_n)}\right]\right)}{\prod_{i=1}^n{\left(1-q^i\right)}}\otimes \mathrm{Fr}\left([V_{\lambda}]\right)\\
                               &= \frac{\mathrm{Fr}\left(\left[S(V^n)^{\mathrm{co}(\mathfrak{S}_n)} \otimes V_{\lambda}\right]\right)}{\prod_{i=1}^n{(1-q^i)}}.
\end{align*}
\end{proof}

\begin{prop}
\label{link_fake_proc}
Take $\lambda \in \mathcal{P}_n$. The following equality holds in $\mathcal{R}(\mathfrak{S}_n)^{\mathrm{gr}}$:
\begin{equation*}
F_{\lambda}(q)\left[\mathscr{P}^n_{\lambda}\right]_{\mathfrak{S}_n}^{\mathrm{gr}} = \left[S(V^n)^{\mathrm{co}(\mathfrak{S}_n)}\otimes V_{\lambda}\right]_{\mathfrak{S}_n}^{\mathrm{gr}}.
\end{equation*}
If, by abuse of notation, one denotes by $\tau_{\ell}$ the irreducible character $\chi_{(n)} \boxtimes \tau_{\ell}$ (where $\chi_{(n)}$ is the trivial character of $\mathfrak{S}_n$), then
\begin{equation*}
F_{\lambda}(\tau_{\ell}) \left[\mathscr{P}^n_{\lambda}\right]_{\mathfrak{S}_n \times \mu_{\ell}} = \left[S(V^n)^{\mathrm{co}(\mathfrak{S}_n)}\otimes V_{\lambda}\right]_{\mathfrak{S}_n \times \mu_{\ell}}
\end{equation*}
in the $\mathbb{Z}$-algebra $\mathcal{R}(\mathfrak{S}_n)\boxtimes \mathcal{R}(\mu_{\ell})$.
\end{prop}
\begin{proof}
Combining Lemma \ref{plethysm_s} with $(\star)$ gives
\begin{align*}
\frac{q^{n(\lambda)}}{\prod_{c \in \mathcal{Y}(\lambda)}{\left(1-q^{h_c(\lambda)}\right)}}\tilde{H}_{\lambda}(z;q,q^{-1}) &=s_{\lambda}\left[\frac{Z}{1-q}\right]\\
                                                                 &= \frac{\mathrm{Fr}\left(\left[S(V^n)^{\mathrm{co}(\mathfrak{S}_n)}\otimes V_{\lambda}\right]\right)}{\prod_{i=1}^n{(1-q^i)}}.
\end{align*}
Combining Lemma \ref{fk_deg} and Proposition \ref{lien_proc_mcdo}, gives
\begin{equation*}
F_{\lambda}(q)\text{ } \mathrm{Fr}\left([\mathscr{P}^n_{\lambda}]^{\mathrm{gr}}\right)=\mathrm{Fr}\left(\left[S(V^n)^{\mathrm{co}(\mathfrak{S}_n)}\otimes V_{\lambda}\right]^{\mathrm{gr}}\right).
\end{equation*}
Taking the inverse Frobenius characteristic map gives the first equality.\\
Since graded modules are the same as $\mathbb{T}_1$-modules, one can take the pullback by ${\tau_{\ell}\colon \mu_{\ell} \to \mathbb{T}_1}$ of the first equality to get the second equality.
\end{proof}

\noindent In the next two subsections, we apply Proposition~\ref{link_fake_proc} to understand the structure of $\mathscr{P}^n_{\lambda}$ as a $(\mathfrak{S}_n\times \mu_{\ell})$-module, and prove directly Corollary~\ref{cor1_a} in two particular cases.

\subsubsection{When $\gamma_{\ell}$ is very small}

Denote by $\mathcal{P}_{n,\ell}^{o}$ the set of all partitions of $n$ with $\ell$-core either empty or equal to $(1)\vdash 1$. We show that Corollary \ref{cor1_a} holds for all $\lambda \in \mathcal{P}_{n,\ell}^o$.

\begin{lemme}
\label{j,l_core}
For each divisor $j$ of $\ell$, the $j$-core of $\lambda$ is equal to the $j$-core of the $\ell$-core of $\lambda$.
\end{lemme}
\begin{proof}
One can use the link between partitions and abacuses \cite[Proposition $3.2$]{Ol93}. Consider the $j$-abacus of $\lambda$. Thanks to \cite[Proposition $1.8$]{Ol93}, one knows that to obtain the $j$-core of $\lambda$, one needs to move, in each runner, all the beads as high as possible. Notice now that with the $j$-abacus one can also obtain the $\ell$-core. Let $\ell=kj$. Again using the result of \cite[Proposition $1.8$]{Ol93}, let us describe a procedure to obtain the $\ell$-core out of the $j$-abacus of $\lambda$. If $i \in \llbracket 0, j\smin 1\rrbracket$, then the level of a position in the $j$-abacus $aj+ i$ is defined to be the integer $a$ and the length of a movement of a bead from a position $a_1j+i$ to a position $a_2j+i$ is defined to be $a_1-a_2$. Now the $\ell$-core of $\lambda$ is obtained by moving all beads, in each runner, as high as possible only with movements of length $k$. One then has that the $j$-core of $\lambda$ is equal to the $j$-core of the $\ell$-core of $\lambda$.
\end{proof}

\begin{lemme}
\label{neq0_0}
For each $\lambda \in \mathcal{P}_{n,\ell}^{o}$, and each $ k \in \llbracket 0, \ell \smin 1 \rrbracket, F_{\lambda}(\zeta_{\ell}^k)\neq 0$.
\end{lemme}
\begin{proof}
Take $j$ a divisor of $\ell$ and denote by $\Phi_j$ the $j$th cyclotomic polynomial. It is then enough to show that $v_{\Phi_j}(F_{\lambda})=0$ where
$v_{\Phi_j}\colon \mathbb{Q}(q) \to \mathbb{Z}$ is the $\Phi_j$-valuation. From  Lemma \ref{fk_deg}, one has
\begin{equation*}
v_{\Phi_j}(F_{\lambda})= \#\left\{i \in \llbracket 1, n \rrbracket \big| i \equiv 0 \text{ } \mathrm{mod} j\right\} - \#\left\{c \in \mathcal{Y}(\lambda) \big| |h_c(\lambda)| \equiv 0 \text{ } \mathrm{mod} j\right\}
\end{equation*}
Now, \cite[Proposition $3.6$]{Ol93} gives the result for $j=\ell$. If $j$ is a divisor of $\ell$, one can again apply \cite[Proposition $3.6$]{Ol93}, by using  Lemma \ref{j,l_core}, to the $j$-core of $\lambda$ which is just the $j$-core of $\gamma_{\ell}$.
\end{proof}

\noindent Recall that $\theta_{\ell}$ denotes the character of $C_{\ell,n}$ such that $\theta_{\ell}(w_{\ell,n}^{\mathrm{g}_{\ell}})=\zeta_{\ell}$.

\begin{prop}
\label{Spr}
If $\lambda \in \mathcal{P}_{n,\ell}^{o}$, then $\left[\mathrm{Res}^{\mathfrak{S}_n}_{C_{\ell,n}}\left(V_{\lambda}\right)\right]=\left[F_{\lambda}\left(\theta_{\ell}^{-1}\right)\right]$.
\end{prop}
\begin{proof}
Consider
\begin{equation*}
v_{\ell,n}:=
\begin{cases}
\left(\zeta_{\ell}^{\ell \smin 1},...,\zeta_{\ell},1,2\zeta_{\ell}^{\ell \smin 1},...,2,...,r_{\ell}\zeta_{\ell}^{\ell \smin 1},...,r_{\ell}\right)\in V^n & \text{ if } \gamma_{\ell}=\emptyset\\
\left(0, \zeta_{\ell}^{\ell \smin 1},...,\zeta_{\ell},1,2\zeta_{\ell}^{\ell \smin 1},...,2,...,r_{\ell}\zeta_{\ell}^{\ell \smin 1},...,r_{\ell}\right)\in V^n & \text{ if } \gamma_{\ell}=(1)\\
\end{cases}.
\end{equation*}
\noindent The stabilizer of $v_{\ell,n}$ in $\mathfrak{S}_n$ is the trivial group. Moreover $v_{\ell,n}$ is an eigenvector of $w_{\ell,n}^{\mathrm{g}_{\ell}}$ with eigenvalue $\zeta_{\ell}$. One can now apply \cite[Proposition $4.5$]{Spr74} to obtain the result.
\end{proof}

\noindent We can now prove Corollary \ref{cor1_a} for all $\lambda \in \mathcal{P}_{n,\ell}^{o}$.

\begin{prop}
For each  partition $\lambda \in \mathcal{P}_{n,\ell}^{o}$,
\begin{equation*}
\left[\mathscr{P}^n_{\lambda}\right]_{\mathfrak{S}_n \times \mu_{\ell}}=\sum_{i=0}^{\ell \smin 1}{\left[\mathrm{Ind}_{C_{\ell,n}}^{\mathfrak{S}_n}\left(\theta_{\ell}^{i}\right)\boxtimes \tau_{\ell}^i\right]_{\mathfrak{S}_n\times \mu_{\ell}}}.
\end{equation*}
\end{prop}
\begin{proof}
Let us start with \cite[Theorem $8$]{MN05} which can be reformulated in the following way
\begin{equation*}
\left[S(V^n)^{\mathrm{co}(\mathfrak{S}_n)}\right]_{\mathfrak{S}_n \times \mu_{\ell}}=\left[\mathscr{P}^n_{(n)}\right]_{\mathfrak{S}_n \times \mu_{\ell}}=\sum_{i=0}^{\ell \smin 1}{\left[\mathrm{Ind}_{C_{\ell,n}}^{\mathfrak{S}_n}\left(\theta_{\ell}^{i}\right)\boxtimes \tau_{\ell}^i\right]_{\mathfrak{S}_n\times \mu_{\ell}}}.
\end{equation*}
Using the second equality of Proposition \ref{link_fake_proc} for $\lambda=(n)$ and Proposition~\ref{Spr}, one obtains
\begin{equation*}
F_{\lambda}(\tau_{\ell})\left[\mathscr{P}^n_{\lambda}\right]_{\mathfrak{S}_n\times \mu_{\ell}} = \sum_{i=0}^{\ell \smin 1}{\left[\mathrm{Ind}_{C_{\ell,n}}^{\mathfrak{S}_n}\left(\theta_{\ell}^iF_{\lambda}(\theta_{\ell}^{-1})\right) \boxtimes \tau_{\ell}^i\right]_{\mathfrak{S}_n\times \mu_{\ell}}}.
\end{equation*}
Let us decompose $F_{\lambda}(\theta)=\sum_{j=0}^{\ell \smin 1}{a_j\theta_{\ell}^j}$ with $a_j \in \mathbb{Z}_{\geq 0}$ and rearrange the two sums
\begin{align*}
F_{\lambda}(\tau_{\ell})\left[\mathscr{P}^n_{\lambda}\right]_{\mathfrak{S}_n\times \mu_{\ell}} =& \sum_{i=0}^{\ell \smin 1}{\sum_{j=0}^{\ell \smin 1}{\left[\mathrm{Ind}_{C_{\ell,n}}^{\mathfrak{S}_n}\left(a_j\theta_{\ell}^{i-j}\right) \boxtimes \tau_{\ell}^i\right]_{\mathfrak{S}_n\times \mu_{\ell}}}}\\
                                                                     =& \sum_{i=0}^{\ell \smin 1}{\sum_{j=0}^{\ell \smin 1}{\left[\mathrm{Ind}_{C_{\ell,n}}^{\mathfrak{S}_n}\left(\theta_{\ell}^{i}\right) \boxtimes a_j\tau_{\ell}^{i+j}\right]_{\mathfrak{S}_n\times \mu_{\ell}}}}\\
                                                                     =& F_{\lambda}(\tau_{\ell})\sum_{i=0}^{\ell \smin 1}{\left[\mathrm{Ind}_{C_{\ell,n}}^{\mathfrak{S}_n}\left(\theta_{\ell}^i\right) \boxtimes \tau_{\ell}^i\right]_{\mathfrak{S}_n\times \mu_{\ell}}}.
\end{align*}
The proposition now follows from Lemma \ref{neq0_0}.
\end{proof}

\subsubsection{When $\gamma_{\ell}$ is small and $\ell$ is prime}

Denote by $\mathcal{P}_n^{<\ell}$ the set of all partitions $\mu$ of $n$ such that the size $\mathrm{g}_{\ell}(\mu)$ of the $\ell$-core of $\mu$ is less than $\ell$. Let us show that Corollary \ref{cor1_a} also holds for all partitions of $\mathcal{P}_n^{<\ell}$ where $\ell$ a prime number.

\begin{prop}
\label{BLM}
For each partition $\lambda$ of n, and each integer $\ell\geq 1$, one has the following equality of $(\mathfrak{S}_n\times \mu_{\ell})$-modules
\begin{equation*}
\left[S(V^n)^{\mathrm{co}(\mathfrak{S}_n)}\right] = \sum_{i=0}^{\ell \smin 1}{\sum_{j=0}^{\ell \smin 1}{\left[\mathrm{Ind}_{W_{\ell,n}^{\mathrm{g}_{\ell}}}^{\mathfrak{S}_n}\left(\left(S(V_{\mathrm{g}_{\ell}})^{\mathrm{co}(\mathfrak{S}_{\mathrm{g}_{\ell}})}\right)^{\ell}_j\boxtimes \theta_{\ell}^{i-j}\right)\boxtimes \tau_{\ell}^i\right]}}.
\end{equation*}
\end{prop}
\begin{proof}
This result is a special case of \cite[Theorem $4.6$]{BLM06}. Let ${\gamma_{\ell}=(\gamma_{\ell,1},...,\gamma_{\ell,t}) \vdash \mathrm{g}_{\ell}}$. Take
\begin{equation*}
v=\left(1,..,1,2,..,2,...,t,...,t,(t+1)\zeta_{\ell}^{\ell \smin 1},...,t+1,...,(t+r_{\ell})\zeta_{\ell}^{\ell \smin 1},...,(t+r_{\ell})\right) \in V^n
\end{equation*}
where $1$ is repeated $\gamma_{\ell,1}$ times, $2$ is repeated $\gamma_{\ell,2}$ times and so on until $t$. The stabilizer of $v$ in $\mathfrak{S}_n$ is exactly $\mathfrak{S}_{\mathrm{g}_{\ell}}$ and $w_{\ell,n}^{\mathrm{g}_{\ell}}v=\zeta_{\ell} v$.
\end{proof}
\noindent For $\lambda \vdash n$, let us denote $a_{\mu,j}^{\ell}(\lambda):= \langle \mathrm{Res}^{\mathfrak{S}_n}_{W_{\ell,n}^{\mathrm{g}_{\ell}}}(V_{\lambda}), V_{\mu} \boxtimes \theta_{\ell}^j \rangle$ where $\mu \vdash \mathrm{g}_{\ell}(\lambda)$, $j \in \llbracket 0, \ell \smin 1 \rrbracket$.
\begin{prop}
\label{spring_gen}
For all partitions $\lambda$ of $n$
\begin{equation*}
F_{\lambda}(\tau_{\ell})=\sum_{\mu \vdash \mathrm{g}_{\ell}}{\sum_{j=0}^{\ell \smin 1}{a_{\mu,j}^{\ell}(\lambda) F_{\mu}(\tau_{\ell}) \tau_{\ell}^{-j}}}.
\end{equation*}
\end{prop}
\begin{proof}
Let us start this proof by showing that
\begin{equation*}
F_{\lambda}(\tau_{\ell}) = \sum_{i=0}^{\ell \smin 1}{\langle V_{(n)}, V_{\lambda}\otimes \left(S(V^n)^{\mathrm{co}(\mathfrak{S}_n)}\right)_i^{\ell} \rangle \tau_{\ell}^i}.
\end{equation*}
Using Proposition \ref{BLM}, one has
\begin{align*}
F_{\lambda}(\tau_{\ell})  &= \sum_{i=0}^{\ell \smin 1}{\sum_{j=0}^{\ell \smin 1}{\langle V_{(n)}, V_{\lambda}\otimes \mathrm{Ind}_{W_{\ell,n}^{\mathrm{g}_{\ell}}}^{\mathfrak{S}_n}\left(\left(S(V_{\mathrm{g}_{\ell}})^{\mathrm{co}(\mathfrak{S}_{\mathrm{g}_{\ell}})}\right)_j^{\ell}\boxtimes \theta_{\ell}^{i-j}\right) \rangle \tau^i_{\ell}}}\\
  &= \sum_{i=0}^{\ell \smin 1}{\sum_{j=0}^{\ell \smin 1}{\sum_{k=0}^{\ell \smin 1}{\sum_{\mu \vdash \mathrm{g}_{\ell}}{a_{\mu,k}^{\ell}(\lambda)\langle V_{(n)}, \mathrm{Ind}_{W_{\ell,n}^{\mathrm{g}_{\ell}}}^{\mathfrak{S}_n}\left(\left(V_{\mu}\otimes\left(S(V_{\mathrm{g}_{\ell}})^{\mathrm{co}(\mathfrak{S}_{\mathrm{g}_{\ell}})}\right)_j^{\ell}\right)\boxtimes \theta_{\ell}^{i-j+k}\right) \rangle \tau^i_{\ell}}}}}.
\end{align*}
One can now use Frobenius reciprocity theorem
\begin{align*}
F_{\lambda}(\tau_{\ell})  &= \sum_{i=0}^{\ell \smin 1}{\sum_{j=0}^{\ell \smin 1}{\sum_{k=0}^{\ell \smin 1}{\sum_{\mu \vdash \mathrm{g}_{\ell}}{a_{\mu,k}^{\ell}(\lambda)\langle V_{(\mathrm{g}_{\ell})}\boxtimes \theta_{\ell}^0, \left(V_{\mu}\otimes\left(S(V_{\mathrm{g}_{\ell}})^{\mathrm{co}(\mathfrak{S}_{\mathrm{g}_{\ell}})}\right)_j^{\ell}\right)\boxtimes \theta_{\ell}^{i-j+k} \rangle \tau^i_{\ell}}}}}\\
    &= \sum_{i=0}^{\ell \smin 1}{\sum_{j=0}^{\ell \smin 1}{\sum_{k=0}^{\ell \smin 1}{\sum_{\mu \vdash \mathrm{g}_{\ell}}{a_{\mu,k}^{\ell}(\lambda)\delta_0^{i-j+k}\langle V_{(\mathrm{g}_{\ell})}, V_{\mu}\otimes\left(S(V_{\mathrm{g}_{\ell}})^{\mathrm{co}(\mathfrak{S}_{\mathrm{g}_{\ell}})}\right)_j^{\ell}\rangle \tau^i_{\ell}}}}}\\
    &= \sum_{j=0}^{\ell \smin 1}{\sum_{k=0}^{\ell \smin 1}{\sum_{\mu \vdash \mathrm{g}_{\ell}}{a_{\mu,k}^{\ell}(\lambda)\langle V_{(\mathrm{g}_{\ell})}, V_{\mu}\otimes\left(S(V_{\mathrm{g}_{\ell}})^{\mathrm{co}(\mathfrak{S}_{\mathrm{g}_{\ell}})}\right)_j^{\ell}\rangle \tau^{j-k}_{\ell}}}}\\
    &= \sum_{\mu \vdash \mathrm{g}_{\ell}}{\sum_{k=0}^{\ell \smin 1}{a_{\mu,k}^{\ell}(\lambda)\sum_{j=0}^{\ell \smin 1}{\langle V_{(\mathrm{g}_{\ell})}, V_{\mu}\otimes\left(S(V_{\mathrm{g}_{\ell}})^{\mathrm{co}(\mathfrak{S}_{\mathrm{g}_{\ell}})}\right)_j^{\ell}\rangle\tau_{\ell}^j \tau^{-k}_{\ell}}}}\\
    &= \sum_{\mu \vdash \mathrm{g}_{\ell}}{\sum_{k=0}^{\ell \smin 1}{a_{\mu,k}^{\ell}(\lambda)F_{\mu}(\tau_{\ell})\tau^{-k}_{\ell}}}.
\end{align*}
\end{proof}

\noindent For the remainder of this subsection, let us suppose that the fixed integer $\ell$ is prime.
\begin{lemme}
\label{neq0_<l}
For each $\lambda \in \mathcal{P}_n^{<\ell}$ and each $k \in \llbracket 0, \ell \smin 1 \rrbracket, F_{\lambda}(\zeta_{\ell}^k) \neq 0$.
\end{lemme}
\begin{proof}
Since $\ell$ is prime it is enough to show that $v_{\Phi_{\ell}}(F_{\lambda})=0$. Using the fact that ${\mathrm{g}_{\ell}<\ell}$, one can use the same argument as in Lemma \ref{neq0_0}.
\end{proof}

\begin{lemme}
\label{equal_not_gl}
Take $\lambda$ a partition of $n$. For all $\mu \in \mathcal{P}_{\mathrm{g}_{\ell}} \setminus \{\gamma_{\ell}\}$,
\begin{equation*}
a_{\mu,j}^{\ell}(\lambda)=a_{\mu,0}^{\ell}(\lambda), \quad \forall j \in \llbracket 0, \ell \smin 1 \rrbracket.
\end{equation*}
\end{lemme}
\begin{proof}
The Murnaghan–Nakayama recursive formula gives the following result
\begin{equation*}
\exists a \in \mathbb{Z}, \forall i \in \llbracket 1, \ell \smin 1 \rrbracket, \forall x \in \mathfrak{S}_{\mathrm{g}_{\ell}}, \chi_{\lambda}\left(x{\left(w_{\ell,n}^{\mathrm{g}_{\ell}}\right)}^i\right)=a\chi_{\gamma_{\ell}}(x)
\end{equation*}
We deduce that:
\begin{equation*}
 \sum_{j=0}^{\ell \smin 1}{a_{\mu,j}^{\ell}(\lambda)\zeta_{\ell}^j=0}, \quad \forall \mu \in \mathcal{P}_{\mathrm{g}_{\ell}}\setminus \{\gamma_{\ell}\}.
\end{equation*}
Indeed,
\begin{align*}
&\sum_{j=0}^{\ell \smin 1}{a_{\mu,j}^{\ell}(\lambda) \zeta_{\ell}^j} = \frac{1}{|W_{\ell,n}^{\mathrm{g}_{\ell}}|}\sum_{j=0}^{\ell \smin 1}{\sum_{i=0}^{\ell\smin 1}{\sum_{x \in \mathfrak{S}_{\mathrm{g}_{\ell}}}{\chi_{\lambda}\left(x{\left(w_{\ell,n}^{\mathrm{g}_{\ell}}\right)}^i\right)\chi_{\mu}(x)\theta_{\ell}^{-ij+j}\left(w_{\ell,n}^{\mathrm{g}_{\ell}}\right)}}}\\
 &= \frac{1}{|W_{\ell,n}^{\mathrm{g}_{\ell}}|}\left(\sum_{j=0}^{\ell \smin 1}{\sum_{x \in \mathfrak{S}_{\mathrm{g}_{\ell}}}{\chi_{\lambda}(x)\chi_{\mu}(x)\theta_{\ell}^j\left(w_{\ell,n}^{\mathrm{g}_{\ell}}\right)}}+ \sum_{j=0}^{\ell \smin 1}{\sum_{i=1}^{\ell \smin 1}{\sum_{x \in \mathfrak{S}_{\mathrm{g}_{\ell}}}{\chi_{\lambda}\left(x{\left(w_{\ell,n}^{\mathrm{g}_{\ell}}\right)}^i\right)\chi_{\mu}(x)\theta_{\ell}^{-ij+j}\left(w_{\ell,n}^{\mathrm{g}_{\ell}}\right)}}}\right)\\
 &= \langle\mathrm{Res}^{\mathfrak{S}_n}_{\mathfrak{S}_{\mathrm{g}_{\ell}}}(\chi_{\lambda}),\chi_{\mu}\rangle\frac{1}{\ell}\sum_{j=0}^{\ell \smin 1}{\theta_{\ell}^j\left(w_{\ell,n}^{\mathrm{g}_{\ell}}\right)} + \frac{a}{|W_{\ell,n}^{\mathrm{g}_{\ell}}|}\sum_{j=0}^{\ell \smin 1}{\sum_{i=1}^{\ell \smin 1}{\sum_{x \in \mathfrak{S}_{\mathrm{g}_{\ell}}}{\chi_{\gamma_{\ell}}(x)\chi_{\mu}(x)\theta_{\ell}^{-ij+j}\left(w_{\ell,n}^{\mathrm{g}_{\ell}}\right)}}}\\
 &= \langle\mathrm{Res}^{\mathfrak{S}_n}_{\mathfrak{S}_{\mathrm{g}_{\ell}}}(\chi_{\lambda}),\chi_{\mu}\rangle\frac{1}{\ell}\sum_{j=0}^{\ell \smin 1}{\theta_{\ell}^j\left(w_{\ell,n}^{\mathrm{g}_{\ell}}\right)} + a\frac{\langle \chi_{\gamma_{\ell}},\chi_{\mu}\rangle}{\ell}\sum_{j=0}^{\ell \smin 1}{\sum_{i=1}^{\ell \smin1}{\theta_{\ell}^{-ij+j}\left(w_{\ell,n}^{\mathrm{g}_{\ell}}\right)}}.
\end{align*}
\noindent The first term is equal to $0$ since $\theta_{\ell}\left(w_{\ell,n}^{\mathrm{g}_{\ell}}\right)=\zeta_{\ell}$ and since $\mu \neq \gamma_{\ell}$, one has $\langle \chi_{\gamma_{\ell}},\chi_{\mu}\rangle = 0$.
Thus $\sum_{j=1}^{\ell \smin 1}{\left(a_{\mu,j}^{\ell}(\lambda)-a_{\mu,0}^{\ell}(\lambda)\right)\zeta_{\ell}^j=0}$ which then gives the result since $\ell$ is prime.
\end{proof}

\noindent One is now able to prove Corollary \ref{cor1_a} in this case.

\begin{prop}
For each $\lambda \in \mathcal{P}_n^{<\ell}$, Corollary \ref{cor1_a} holds.
\end{prop}
\begin{proof}
For a given finite group $G$, let $\mathrm{Reg}(G)$ denote the regular representation of $G$.  We wish to show the following equality of $(\mathfrak{S}_n \times \mu_{\ell})$-modules
\begin{equation*}
\left[\mathscr{P}^n_{\lambda}\right]=\sum_{i=0}^{\ell \smin 1}{\sum_{j=0}^{\ell \smin 1}{\left[\mathrm{Ind}_{W_{\ell,n}^{\mathrm{g}_{\ell}}}^{\mathfrak{S}_n}\left(\left(\mathscr{P}^{\mathrm{g}_{\ell}}_{\gamma_{\ell}}\right)^{\ell}_j\boxtimes\theta_{\ell}^{i-j}\right)\boxtimes \tau_{\ell}^i\right]}}.
\end{equation*}
Using Proposition \ref{BLM}, the right-hand side of the second equality of Proposition \ref{link_fake_proc} can be rewritten as
\begin{equation}
\label{rewritenRHS}
\sum_{\mu \vdash \mathrm{g}_{\ell}}{\sum_{k=0}^{\ell \smin 1}{\sum_{j=0}^{\ell \smin 1}{\sum_{i=0}^{\ell \smin 1}{a_{\mu,k}^{\ell}(\lambda)\left[\mathrm{Ind}^{\mathfrak{S}_n}_{W_{\ell,n}^{\mathrm{g}_{\ell}}}\left(\left(V_{\mu}\otimes \left(S(V_{\mathrm{g}_{\ell}})^{\mathrm{co}(\mathfrak{S}_{\mathrm{g}_{\ell}})}\right)^{\ell}_j\right)\boxtimes \theta_{\ell}^{i-j+k}\right)\boxtimes \tau_{\ell}^i\right]}}}}.
\end{equation}
Let us fix  $\mu \in \mathcal{P}_{\mathrm{g}_{\ell}}\setminus \{\gamma_{\ell}\}$ and consider the associated term in (\ref{rewritenRHS}). Using Lemma \ref{equal_not_gl}, this term is equal to
\begin{equation*}
a_{\mu,0}^{\ell}(\lambda)\sum_{j=0}^{\ell \smin 1}{\left[\mathrm{Ind}^{\mathfrak{S}_n}_{W_{\ell,n}^{\mathrm{g}_{\ell}}}\left(\left(V_{\mu}\otimes \left(S(V_{\mathrm{g}_{\ell}})^{\mathrm{co}(\mathfrak{S}_{\mathrm{g}_{\ell}})}\right)^{\ell}_j\right)\boxtimes \mathrm{Reg}(C_{\ell,n}) \right) \boxtimes \mathrm{Reg}(\mu_{\ell})\right]}.
\end{equation*}
Denote for all $\nu \vdash \mathrm{g}_{\ell}, F_{\nu}(\tau_{\ell})=\sum_{k=0}^{\ell \smin 1}{f_{\nu,k}\tau_{\ell}^k}$. Applying the second equality of Proposition \ref{link_fake_proc} for $\mu$, gives us
\begin{equation*}
a_{\mu,0}^{\ell}(\lambda)\sum_{k=0}^{\ell \smin 1}{\sum_{j=0}^{\ell \smin 1}{f_{\mu,k}\left[\mathrm{Ind}^{\mathfrak{S}_n}_{W_{\ell,n}^{\mathrm{g}_{\ell}}}\left(\left(\mathscr{P}^{\mathrm{g}_{\ell}}_{\mu}\right)^{\ell}_{j-k}\boxtimes \mathrm{Reg}(C_{\ell,n}) \right)\boxtimes \mathrm{Reg}(\mu_{\ell})\right]}}.
\end{equation*}
By construction of the Procesi bundle, $\sum_{j=0}^{\ell \smin 1}{\left[(\mathscr{P}^{\mathrm{g}_{\ell}}_{\nu})^{\ell}_j\right]}=\left[\mathrm{Reg}(\mathfrak{S}_{\mathrm{g}_{\ell}})\right]$ and by definition of the fake degree $\sum_{j=0}^{\ell \smin 1}{f_{\nu,k}}=\dim(V_{\nu})$. Summing everything up leads to
\begin{equation*}
a_{\mu,0}^{\ell}(\lambda)\sum_{k=0}^{\ell \smin 1}{\sum_{j=0}^{\ell \smin 1}{f_{\mu,k}\left[\mathrm{Ind}^{\mathfrak{S}_n}_{W_{\ell,n}^{\mathrm{g}_{\ell}}}\left(\left(\mathscr{P}^{\mathrm{g}_{\ell}}_{\mu}\right)^{\ell}_{j-k}\boxtimes \mathrm{Reg}\left(C_{\ell,n}\right)\right)\boxtimes \mathrm{Reg}(\mu_{\ell})\right]}},
\end{equation*}
which is equal to
\begin{equation*}
a_{\mu,0}^{\ell}(\lambda)\dim(V_{\mu})\left[\mathrm{Reg}(\mathfrak{S}_n\times \mu_{\ell})\right].
\end{equation*}

\noindent The last equality holds for the fiber of the Procesi bundle over $I_{\mu}$ for any partition $\mu$ of $\mathrm{g}_{\ell}$. In particular, it holds for $I_{\gamma_{\ell}}$. One gets that the term
\begin{equation*}
\sum_{i=0}^{\ell \smin 1}{\sum_{j=0}^{\ell \smin 1}{\sum_{k=0}^{\ell \smin 1}{a_{\mu,k}^{\ell}(\lambda)\left[\mathrm{Ind}^{\mathfrak{S}_n}_{W_{\ell,n}^{\mathrm{g}_{\ell}}}\left(\left(V_{\mu}\otimes \left(S(V_{\mathrm{g}_{\ell}})^{\mathrm{co}(\mathfrak{S}_{\mathrm{g}_{\ell}})}\right)^{\ell}_j\right)\boxtimes \theta_{\ell}^{i-j+k}\right) \boxtimes \tau_{\ell}^i\right]}}}
\end{equation*}
is equal to
\begin{equation*}
a_{\mu,0}^{\ell}(\lambda)\sum_{k=0}^{\ell \smin 1}{\sum_{j=0}^{\ell \smin 1}{f_{\mu,k}\left[\mathrm{Ind}^{\mathfrak{S}_n}_{W_{\ell,n}^{\mathrm{g}_{\ell}}}\left(\left(\mathscr{P}^{\mathrm{g}_{\ell}}_{\gamma_{\ell}}\right)^{\ell}_{j-k} \boxtimes \mathrm{Reg}(C_{\ell,n})\right)\boxtimes \mathrm{Reg}(\mu_{\ell})\right]}},
\end{equation*}
which can be rewritten as
\begin{equation*}
\sum_{i=0}^{\ell \smin 1}{\sum_{j=0}^{\ell \smin 1}{\sum_{k=0}^{\ell \smin 1}{a_{\mu,k}^{\ell}(\lambda)\left[\mathrm{Ind}^{\mathfrak{S}_n}_{W_{\ell,n}^{\mathrm{g}_{\ell}}}\left(\left(\mathscr{P}^{\mathrm{g}_{\ell}}_{\gamma_{\ell}}\right)^{\ell}_j\boxtimes \left(\theta_{\ell}^{i-j}F_{\mu}\left(\theta_{\ell}^{-1}\right)\theta_{\ell}^{k}\right)\right)\boxtimes \tau_{\ell}^i\right]}}}.
\end{equation*}
Finally, for $\mu=\gamma_{\ell}$,
\begin{equation*}
\sum_{i=0}^{\ell \smin 1}{\sum_{j=0}^{\ell \smin1}{\sum_{k=0}^{\ell \smin1}{a_{\gamma_{\ell},k}^{\ell}(\lambda)\left[\mathrm{Ind}^{\mathfrak{S}_n}_{W_{\ell,n}^{\mathrm{g}_{\ell}}}\left(\left(V_{\gamma_{\ell}}\otimes \left(S(V_{\mathrm{g}_{\ell}})^{\mathrm{co}(\mathfrak{S}_{\mathrm{g}_{\ell}})}\right)^{\ell}_j\right)\boxtimes \theta_{\ell}^{i-j+k}\right)\boxtimes \tau_{\ell}^i\right]}}}
\end{equation*}
is equal to
\begin{equation*}
\sum_{i=0}^{\ell \smin 1}{\sum_{j=0}^{\ell \smin 1}{\sum_{k=0}^{\ell \smin1}{a_{\gamma_{\ell},k}^{\ell}(\lambda)\left[\mathrm{Ind}^{\mathfrak{S}_n}_{W_{\ell,n}^{\mathrm{g}_{\ell}}}\left(\left(\mathscr{P}^{\mathrm{g}_{\ell}}_{\gamma_{\ell}}\right)^{\ell}_j\boxtimes \left(\theta_{\ell}^{i-j}F_{\gamma_{\ell}}\left(\theta_{\ell}^{-1}\right)\theta_{\ell}^{k}\right)\right)\boxtimes \tau_{\ell}^i\right]}}}.
\end{equation*}

\noindent By putting the pieces back together, and using Proposition \ref{spring_gen}, one gets
\begin{equation*}
\sum_{i=0}^{\ell \smin1}{\sum_{j=0}^{\ell \smin1}{\left[\mathrm{Ind}^{\mathfrak{S}_n}_{W_{\ell,n}^{\mathrm{g}_{\ell}}}\left(\left(\mathscr{P}^{\mathrm{g}_{\ell}}_{\gamma_{\ell}}\right)^{\ell}_j\boxtimes \left(\theta_{\ell}^{i-j}F_{\lambda}\left(\theta_{\ell}^{-1}\right)\right)\right)\boxtimes \tau_{\ell}^i\right]}},
\end{equation*}
is equal to
\begin{equation*}
F_{\lambda}(\tau_{\ell})\sum_{i=0}^{\ell \smin1}{\sum_{j=0}^{\ell \smin 1}{\left[\mathrm{Ind}^{\mathfrak{S}_n}_{W_{\ell,n}^{\mathrm{g}_{\ell}}}\left(\left(\mathscr{P}^{\mathrm{g}_{\ell}}_{\gamma_{\ell}}\right)^{\ell}_j\boxtimes \theta_{\ell}^{i-j}\right)\boxtimes \tau_{\ell}^i\right]}},
\end{equation*}
which completes the proof, after applying Lemma \ref{neq0_<l}.
\end{proof}

\section{Combinatorial consequences in type $D$}

In this last section, we consider the case where $\Gamma$ is of type $D$. Let us fix an integer $\ell=4l \geq 1$ divisible by $4$. Let $s$ be an element of $\mathrm{SL}_2(\mathbb{C})$ with integer coefficients and diagonal coefficients equal to $0$ and denote $BD_{\ell}$ the finite subgroup of $\mathrm{SL}_2(\mathbb{C})$ generated by $\omega_{2l}$ and $s$. The order of $BD_{\ell}$ is $\ell$. Recall that $\tau_{2l}$ is the character of $\mu_{2l}$ that maps $\omega_{2l}$ to $\zeta_{2l}$. For $i \in \mathbb{Z}$, consider $\chi_i:=\mathrm{Ind}_{\mu_{2l}}^{BD_{\ell}}{\tau_{2l}^i}$. Note that $\chi_i$ is irreducible if and only if $i$ is not congruent to $0$ or $l$ modulo $2l$. If $l$ is even, the character table of $BD_{\ell}$ is

\begin{center}
\begin{tabular}{|c |c |>{\centering\arraybackslash}c|>{\centering\arraybackslash}c |>{\centering\arraybackslash}c|>{\centering\arraybackslash}c|}
\hline
order & $1$ & $1$ & $2$ & $l$ & $l$ \\
\hline
classes & $\begin{pmatrix} 1 & 0 \\ 0 & 1 \end{pmatrix}$ & $\begin{pmatrix} \smin 1 & 0 \\ 0 & \smin 1 \end{pmatrix}$ & $\omega_{2l}^p (0 < p < l)$ & $s$ & $s\omega_{2l}$ \\
\hhline{|=|=|=|=|=|=|}
$\chi_{0^+}$ & $1$ & $1$ & $1$ & $1$ & $1$\\
\hline
$\chi_{0^-}$ & $1$ & $1$ & $1$ & $\smin 1$ & $\smin 1$\\
\hline
$\chi_{l^+}$ & $1$ & $1$ & $(\smin 1)^p$ & $\smin 1$ & $1$\\
\hline
$\chi_{l^-}$ & $1$ & $1$ & $(\smin 1)^p$ & $1$ & $\smin 1$\\
\hline
\thead{$\chi_k$ \\ $( 0<k<l)$}  & $2$ & $(\smin 1)^k2$ & $2 \mathrm{cos}(\frac{kp\pi}{l})$ & $0$ & $0$\\
\hline
\end{tabular}
\end{center}

\noindent If $l$ is odd, the character table of $BD_{\ell}$ is

\begin{center}
\begin{tabular}{|c |c |>{\centering\arraybackslash}c|>{\centering\arraybackslash}c |>{\centering\arraybackslash}c|>{\centering\arraybackslash}c|}
\hline
order & $1$ & $1$ & $2$ & $l$ & $l$ \\
\hline
classes & $\begin{pmatrix} 1 & 0 \\ 0 & 1 \end{pmatrix}$ & $\begin{pmatrix} \smin 1 & 0 \\ 0 & \smin 1 \end{pmatrix}$ & $\omega_{2l}^p (0 < p < l)$ & $s$ & $s\omega_{2l}$ \\
\hhline{|=|=|=|=|=|=|}
$\chi_{0^+}$ & $1$ & $1$ & $1$ & $1$ & $1$\\
\hline
$\chi_{0^-}$ & $1$ & $1$ & $1$ & $\smin 1$ & $\smin 1$\\
\hline
$\chi_{l^+}$ & $1$ & $\smin1$ & $(\smin1)^p$ & $\zeta_4$ & $\smin\zeta_4$\\
\hline
$\chi_{l^-}$ & $1$ & $\smin1$ & $(\smin1)^p$ & $\smin\zeta_4$ & $\zeta_4$\\
\hline
\thead{$\chi_k$ \\ $( 0<k<l)$}  & $2$ & $(\smin1)^k2$ & $2 \mathrm{cos}(\frac{kp\pi}{l})$ & $0$ & $0$\\
\hline
\end{tabular}
\end{center}

\noindent Thanks to \cite[Theorem $5.25$]{Pae24}, the irreducible components of $\mathcal{H}_n^{BD_{\ell}}$ that contain a $\mathbb{T}_1$-fixed point are parametrized by symmetric $2l$-cores of size less than or equal to $n$ and congruent to $n$ modulo $2l$. Moreover, the fixed points of $\mathcal{H}_n$ under $\langle \mathbb{T}_1,BD_{\ell} \rangle$ are the monomial ideals parametrized by symmetric partitions of $n$. When $\lambda$ is a symmetric partition of $n$, the fiber of the Procesi bundle over $I_{\lambda}$ is then an $(\mathfrak{S}_n \times BD_{\ell})$-module. As such, it admits a decomposition
\begin{equation*}
\left[\mathscr{P}^n_{\lambda}\right]_{\mathfrak{S}_n\times BD_{\ell}}=\sum_{\chi \in I_{BD_{\ell}}}{\left[D_{n,\chi}^{\ell}(\lambda)\boxtimes \chi\right]_{\mathfrak{S}_n \times BD_{\ell}}},
\end{equation*}
where $I_{BD_{\ell}}$ is the set of irreducible characters of $BD_{\ell}$. The goal of this section is to describe the $\mathfrak{S}_n$-modules $D_{n,\chi}^{\ell}(\lambda)$ for each $\chi \in I_{BD_{\ell}}$. To do so, we will use  Corollary \ref{cor1_a}.
\begin{lemme}
\label{4l_mult}
If $\lambda$ is a symmetric partition of $n$, then the number ${r_{D,2l}(\lambda):=\frac{n-\mathrm{g}_{2l}(\lambda)}{2l}}$ is a multiple of $2$.
\end{lemme}
\begin{proof}
To prove this one can use \cite[Lemma $2.2$]{Ol93} and the link between abacuses and $\beta$-sets to see that the $2l$-abacus of $\lambda^*$ is equal to the horizontal reflection of the $2l$-abacus of $\lambda$. When $\lambda$ is symmetric, the number $r_{D,2l}$ of $2l$-hooks that needs to be removed to go from $\lambda$ to $\mathrm{g}_{2l}$ is a multiple of $2$.
\end{proof}

\begin{lemme}
\label{res_d_a}
The restrictions of the irreducible characters of $BD_{\ell}$ to $\mu_{2l}$ are
\begin{multicols}{2}
    \begin{itemize}
	\item $\mathrm{Res}^{BD_{\ell}}_{\mu_{2l}}(\chi_{0^+})=\tau_{2l}^0$
	\item $\mathrm{Res}^{BD_{\ell}}_{\mu_{2l}}(\chi_{0^-})=\tau_{2l}^0$
        \item  $\mathrm{Res}^{BD_{\ell}}_{\mu_{2l}}(\chi_{l^+})=\tau_{2l}^{\ell}$
        \item  $\mathrm{Res}^{BD_{\ell}}_{\mu_{2l}}(\chi_{l^-})=\tau_{2l}^{\ell}$
    \end{itemize}
    \end{multicols}
\begin{itemize}
\addtolength{\itemindent}{2.5cm}
\item $\forall i \in \llbracket 1, l \smin 1 \rrbracket,\text{ } \mathrm{Res}^{BD_{\ell}}_{\mu_{2l}}(\chi_i)=\tau_{2l}^i+\tau_{2l}^{-i}$
\end{itemize}
\end{lemme}

\noindent From there one can deduce the following information on the $D_{n,\chi}^{\ell}(\lambda)$ modules.
\begin{prop}
\label{link_a_d}
For each symmetric partition $\lambda$ of $n$, the following equalities hold in $\mathcal{R}(\mathfrak{S}_n)$
\begin{enumerate}[label=(\roman*)]
\item $\left[D_{n,\chi_{0^+}}^{\ell}(\lambda)+D_{n,\chi_{0^-}}^{\ell}(\lambda)\right]=\left[\left(\mathscr{P}^n_{\lambda}\right)_{0}^{2l}\right]$
\item $\left[D_{n,\chi_{l^+}}^{\ell}(\lambda)+D_{n,\chi_{l^-}}^{\ell}(\lambda)\right]=\left[\left(\mathscr{P}^n_{\lambda}\right)_{l}^{2l}\right]$
\item $\left[D_{n,\chi_{l^+}}^{\ell}(\lambda)\right]=\left[D_{n,\chi_{l^-}}^{\ell}(\lambda)\right]$
\item $\forall i \in \llbracket 1, l \smin 1 \rrbracket, \left[D_{n,\chi_{i}}^{\ell}(\lambda)\right]=\left[\left(\mathscr{P}^n_{\lambda}\right)_{i}^{2l}\right]=\left[\left(\mathscr{P}^n_{\lambda}\right)_{2l-i}^{2l}\right]$
\end{enumerate}
\end{prop}
\begin{proof}
Equality $(i)$ and $(ii)$ comes directly from Lemma \ref{res_d_a}. Concerning equality $(iii)$, note that $BD_{\ell} \triangleleft BD_{2\ell}$ and that $\omega_{\ell}$ acts nontrivially on $I_{BD_{\ell}}$. It swaps $\chi_{l^+}$ and $\chi_{l^-}$ and fixes all other irreducible characters of $BD_{\ell}$. The $\mathfrak{S}_n$-module $\mathscr{P}^n_{\lambda}$ being bigraded, it follows that
\begin{equation*}
\omega_{\ell}.\left[\mathscr{P}^n_{\lambda}\right]_{\mathfrak{S}_n\times BD_{2\ell}}=\left[\mathscr{P}^n_{\lambda}\right]_{\mathfrak{S}_n \times BD_{2\ell}}.
\end{equation*}
Now, applying the restriction from $\mathfrak{S}_n \times BD_{2\ell}$ to $\mathfrak{S}_n \times BD_{\ell}$, one has
\begin{equation*}
\left[D_{n,\chi_{l^+}}^{\ell}(\lambda)\right]=\left[D_{n,\chi_{l^-}}^{\ell}(\lambda)\right].
\end{equation*}

\noindent Moreover, by combining \cite[Proposition $3.5.11$]{H03} with Lemma \ref{res_d_a}, it follows that
\begin{equation*}
2\left[D_{n,\chi_{i}}^{\ell}(\lambda)\right]=\left[\left(\mathscr{P}_{\lambda}^n\right)_{i}^{2l}+\left(\mathscr{P}_{\lambda}^n\right)_{2l-i}^{2l}\right]=2\left[\left(\mathscr{P}_{\lambda}^n\right)_{i}^{2l}\right].
\end{equation*}
\end{proof}

\noindent Lemma \ref{4l_mult} implies that $n-\mathrm{g}_{2l}$ is a multiple of $\ell = 4l$. Recall from \eqref{eq:wperm} that we can associate to the integers $n,\mathrm{g}_{2l}$ and $r_{2l} = (n - \mathrm{g}_{2l})/2l$ the permutation $w^{\mathrm{g}_{2l}}_{2l,n} \in \mathfrak{S}_n$ of order $2l$. Moreover, one can choose $s_{2l} \in \mathfrak{S}_{n-\mathrm{g}_{2l}}$ such that $N_{\ell} := \langle w^{\mathrm{g}_{2l}}_{2l,n},s_{2l} \rangle \subset \mathfrak{S}_{n - \mathrm{g}_{2l}}$ is abstractly isomorphic to $BD_{\ell}$.

\begin{ex}
When $l=2$, $\lambda$ is a symmetric partition of $8$, $r_{D,4}(\lambda)=2$ then $\mathrm{g}_{2l}=0$. In that case $w^{0}_{2l,n}=(1 2 3 4)(5 6 7 8) \in \mathfrak{S}_8$ and one can take $s_{2l}=(1 8 3 6)(2 7 4 5) \in \mathfrak{S}_8$.
\end{ex}

\begin{prop}
\label{prop_d_type}
 For each symmetric partition $\lambda$ of $n$ and for each $i \in \llbracket 1, l \smin 1 \rrbracket$,
\begin{equation*}
\left[D^{\ell}_{n,\chi_i}(\lambda)\right]_{\mathfrak{S}_n}=\sum_{j=0}^{2l \smin 1}{\left[\mathrm{Ind}_{\mathfrak{S}_{\mathrm{g}_{2l}}\times N_{\ell}}^{\mathfrak{S}_n}\left(\left(\mathscr{P}_{\gamma_{2l}}^{\mathrm{g}_{2l}}\right)^{2l}_j\boxtimes \chi_{i-j}\right)\right]_{\mathfrak{S}_n}}.
\end{equation*}
Moreover,
\begin{equation*}
\left[D^{\ell}_{n,\chi_{l^+}}(\lambda)\right]_{\mathfrak{S}_n}=\frac{1}{2}\sum_{j=0}^{2l \smin 1}{\left[\mathrm{Ind}_{\mathfrak{S}_{\mathrm{g}_{2l}}\times N_{\ell}}^{\mathfrak{S}_n}\left(\left(\mathscr{P}_{\gamma_{2l}}^{\mathrm{g}_{2l}}\right)^{2l}_j\boxtimes \chi_{l-j}\right)\right]_{\mathfrak{S}_n}}.
\end{equation*}
\end{prop}
\begin{proof}
Fix $i \in \llbracket 1,l \smin 1 \rrbracket$. Thanks to Proposition \ref{link_a_d} and Corollary \ref{cor1_a}, one has
\begin{align*}
\left[D^{\ell}_{n,\chi_i}(\lambda)\right]_{\mathfrak{S}_n}=&\sum_{j=0}^{2l \smin1}{\left[\mathrm{Ind}_{W^{\mathrm{g}_{2l}}_{2l,n}}^{\mathfrak{S}_n}\left(\left(\mathscr{P}_{\gamma_{2l}}^{\mathrm{g}_{2l}}\right)^{2l}_j\boxtimes \theta_{2l}^{i-j}\right)\right]_{\mathfrak{S}_n}}\\
                      =&\sum_{j=0}^{2l \smin 1}{\left[\mathrm{Ind}_{\mathfrak{S}_{\mathrm{g}_{2l}}\times N_{\ell}}^{\mathfrak{S}_n}\left(\mathrm{Ind}_{W_{2l,n}^{\mathrm{g}_{2l}}}^{\mathfrak{S}_{\mathrm{g}_{2l}}\times N_{\ell}}\left(\left(\mathscr{P}_{\gamma_{2l}}^{\mathrm{g}_{2l}}\right)^{2l}_j\boxtimes \theta_{2l}^{i-j}\right)\right)\right]_{\mathfrak{S}_n}}\\
                      =&\sum_{j=0}^{2l \smin 1}{\left[\mathrm{Ind}_{\mathfrak{S}_{\mathrm{g}_{2l}}\times N_{\ell}}^{\mathfrak{S}_n}\left(\left(\mathscr{P}_{\gamma_{2l}}^{\mathrm{g}_{2l}}\right)^{2l}_j\boxtimes \chi_{i-j}\right)\right]_{\mathfrak{S}_n}}.
\end{align*}
The same computation gives us the second formula for $\left[D^{\ell}_{n,\chi_{l^+}}(\lambda)\right]_{\mathfrak{S}_n}$. Thanks to Proposition \ref{link_a_d}, it is equal to $\frac{1}{2}\left[\left(\mathscr{P}^n_{\lambda}\right)_{l}^{2l}\right]_{\mathfrak{S}_n}$.
\end{proof}

\begin{rmq}
Note that Proposition~\ref{link_a_d} and Proposition~\ref{prop_d_type} together allow one to express all but two of the $\mathfrak{S}_n$-modules $\left[D^{\ell}_{n,\chi}(\lambda)\right]$ in terms of the $\mathfrak{S}_{\mathrm{g}_{2l}}$-modules $\left(\left[D^{\ell}_{\mathrm{g}_{2l},\chi}(\lambda)\right]\right)_{\chi \in I_{BD_{\ell}}}$. It is not clear how to express $\left[D_{n,\chi_{0^{+}}}^{\ell}(\lambda)\right]$ and $\left[D_{n,\chi_{0^{-}}}^{\ell}(\lambda)\right]$ in this way since we do not know how $\left[\left(\mathscr{P}^n_{\lambda}\right)_{0}^{2l}\right]$ splits in two in Proposition~\ref{link_a_d}(i).
\end{rmq}
\printbibliography
\Addresses
\end{document}